\documentclass[fleqn]{elsarticle}
\usepackage{amsmath,amsfonts,graphicx,subfigure,epstopdf,float,amssymb,bm,amsthm,multirow,appendix,bbm,color,enumitem}

\usepackage{geometry}
\usepackage{geometry}
\usepackage[colorlinks,linkcolor=red,citecolor=blue]{hyperref}
\usepackage{booktabs,multirow}
\numberwithin{equation}{section}
\newtheorem{thm}{Theorem}[section]

\newtheorem{rmk}{Remark}[section]
\newtheorem{scm}{Scheme}[section]

\graphicspath{{FIGS/}}

\def\wh{\widehat}
\def\wt{\widetilde}
\def\sech{\mathrm{sech}}

\allowdisplaybreaks[4] 

\begin{document}

\begin{frontmatter}

\title{A novel class of energy-preserving Runge-Kutta methods for the Korteweg-de Vries equation}

\author[mymainaddress,mythirdaddress]{Yue Chen}
\author[mymainaddress,mythirdaddress,mysecondaddress]{Yuezheng Gong}
\cortext[mycorrespondingauthor]{Corresponding author}
\ead{gongyuezheng@nuaa.edu.cn}
\author[mymainaddress,mythirdaddress,mysecondaddress]{Qi Hong}
\author[mymainaddress,mythirdaddress]{Chunwu Wang}
\address[mymainaddress]{Department of Mathematics, Nanjing University of Aeronautics and Astronautics, Nanjing 211106, China}
\address[mythirdaddress]{Key Laboratory of Mathematical Modelling and High Performance Computing of Air Vehicles (NUAA), MIIT, Nanjing  211106, China}
\address[mysecondaddress]{Jiangsu Key Laboratory for Numerical Simulation of Large Scale Complex Systems, Nanjing  210023, China}
	
	\begin{abstract}
		In this paper, we present a quadratic auxiliary variable approach to develop a new class of energy-preserving Runge-Kutta methods for the Korteweg-de Vries equation. The quadratic auxiliary variable approach is first proposed to reformulate the original model into an equivalent system, which transforms the energy conservation law of the Korteweg-de Vries equation into two quadratic invariants of the reformulated system. Then the symplectic Runge-Kutta methods are directly employed for the reformulated model to arrive at a new kind of time semi-discrete schemes for the original problem. Under the consistent initial condition, the proposed methods are rigorously proved to maintain the original energy conservation law of the Korteweg-de Vries equation. In addition, the Fourier pseudo-spectral method is used for spatial discretization, resulting in fully discrete energy-preserving schemes. To implement the proposed methods effectively, we present a very efficient iterative technique, which not only greatly saves the calculation cost, but also achieves the purpose of practically preserving structure. Ample numerical results are addressed to confirm the expected order of accuracy, conservative property and efficiency of the proposed algorithms. 
	\end{abstract}
	
	
\end{frontmatter}

\section{Introduction}
	In this paper, we are concerned with the Korteweg-de Vries (KdV) equation
	\begin{equation}\label{KdVeq}
		u_t+\eta u u_x+\mu^2u_{xxx}=0, \quad (x,t)\in [a,b]\times (0,T],
	\end{equation}
	with periodic boundary condition
	\begin{equation}\label{BC}
		u(a,t) = u(b,t), \quad t\in  [0,T],
	\end{equation}
	and initial condition
	\begin{equation}\label{IC}
		u(x,0) = u_0(x),\quad  x\in  [a,b],
	\end{equation}
	where $\eta$ and $\mu$ are two real parameters. It is an important nonlinear hyperbolic equation with smooth solution at all times and also a mathematical waves on shallow water surfaces. Eq. \eqref{KdVeq} has been used to describe various phenomena such as waves in bubble-liquid mixtures, acoustic waves in an anharmonic crystal, magnetohydrodynamic waves in warm plasma and ion acoustic waves \cite{1965Interaction}.

In the past half century, numerous numerical methods have been developed for the KdV equation, including Galerkin methods \cite{WintherMC1980,BonaFully1986,Yan2002A,Xu2007Error}, finite difference schemes \cite{Ascher2004Multisymplectic,Zhao2000Multisymplectic}, Fourier spectral or pseudo-spectral methods \cite{Gong2014Some,Brugnano2019Energy}, operator splitting and exponential-type integrators \cite{Holden2011Operator,Martina2016An}, etc. Recently, there has been a surge on constructing numerical methods for dynamical systems governed by differential equations to preserve as many properties of the continuous system as possible. Numerical methods that preserve at least some of the structural properties of the continuous dynamical system are called geometric integrators or structure-preserving algorithms \cite{Feng2003book,2006Geometric}. Many geometric integrators have been presented for the KdV equation, especially the symplectic and multisymplectic schemes \cite{Huang1991A,Zhao2000Multisymplectic,Ascher2004Multisymplectic,Wang2004Numerical}. In recent years, various energy-preserving and momentum-preserving algorithms have been developed for this equation as well \cite{Frutos1997Accuracy,Cui2007Numerical,Karasozen2013Energy}. More recently, some local structure-preserving algorithms, originally discussed by Wang et al. \cite{Wang2008Local}, have been applied for the KdV equation \cite{Gong2014Some,Wang2017Local}. However, most of the existing structure-preserving algorithms are only up to second order in time, which cannot usually provide long time accurate solutions with a given large time step.

As a matter of fact, how to devise high-order invariant-preserving methods for conservative systems has attracted a lot of attention in recent years. It is well known that all Runge-Kutta (RK) methods preserve linear invariants, while only those that satisfy the symplectic condition conserve all quadratic invariants \cite{Cooper1987Stability}. For canonical Hamiltonian systems, many high-order energy-preserving algorithms have been developed, including high-order averaged vector field (AVF) methods \cite{Quispel2008A,Wang2012A,Li2016A}, Hamiltonian Boundary Value Methods (HBVMs) \cite{Brugnano2010Hamiltonian}, energy-preserving variant of collocation methods \cite{Hairer2010Energy} and time finite element methods \cite{Tang2012Time}, etc. In addition, the above mentioned high-order energy-preserving methods are also valid for Hamiltonian systems with constant skew-symmetric structural matrix. For general conservative systems, these methods should be further discussed (e.g., see \cite{Cohen2011Linear,Brugnano2012Energy,Brugnano2020Arbitrarily}). As far as we know, the HBVMs have been applied for the KdV equation to obtain high-order energy-preserving methods \cite{Song2017Hamiltonian,Brugnano2019Energy}. It should be noted that all of these methods involve integrals, which often need to be replaced by high-precision numerical integration formulas for practicality. Therefore, these methods can exactly conserve polynomial energy, while they can only be practically energy-preserving for non-polynomial cases \cite{Brugnano2016Line}.

Recently, the invariant energy quadratization (IEQ) \cite{Yang2017Numerical} and the scalar auxiliary variable (SAV) approaches \cite{Shen2018The}, originally proposed for gradient flow models, have been successfully applied for various conservative systems \cite{Jiang2019AJSC,Jiang2020AJSC,Cai2020Two}. Based on these techniques, many high-order structure-preserving algorithms have been developed for various models, including dissipative systems \cite{Akrivis2019Energy,GongSISC2020ARBITRARILY,Gong2020ArbitrarilyJCP} and conservative systems \cite{Jiang2020Arbitrarily,Zhang2020Novel,Jiang2021Explicit}. However, different from traditional structure-preserving algorithms, these numerical strategies only maintain a modified quadratic energy, which may not be the essential property of the original model. 

In this paper, we propose a new numerical strategy to develop arbitrarily high-order energy-preserving algorithms for general conservative systems with a polynomial energy. We first present a quadratic auxiliary variable (QAV) approach to reformulate the original model into an equivalent system, which transforms the energy conservation law of the original problem into two quadratic invariants of the reformulated system. It is important to note that if a numerical method can preserve the two quadratic invariants of the new system, it will retain the original energy conservation law. Fortunately, symplectic RK methods can conserve all quadratic invariants, so they are used directly for the reformulated system to develop a novel class of energy-preserving algorithms for the original model.
Under the consistent initial condition, the new proposed methods are rigorously proved to preserve the original energy conservation law. For the sake of clarity, we will take the KdV equation as an example to illustrate the idea. Furthermore, the Fourier pseudo-spectral method is employed for developing spatial structure-preserving discretization, resulting in fully discrete high-order energy-preserving schemes. In addition, we provide a very efficient iterative technique to solve the proposed nonlinear schemes, which not only greatly saves the computing cost, but also achieves the purpose of practically preserving structure. Numerical experiments are presented to demonstrate the accuracy, conservative property and efficiency of the proposed methods. 

The rest of this paper is organized as follows. In section 2, we present the QAV approach to reformulate the KdV equation and discuss the structure-preserving properties of the reformulated system. In section 3, we propose a class of high-order energy-preserving schemes based on the QAV reformulation. In section 4, the Fourier pseudo-spectral method is employed to give rise to the spatial discretization. A practically structure-preserving iterative technique is developed in section 5. Numerical examples are shown to validate the accuracy and efficiency  of the proposed schemes in section 6. Finally, we give some conclusions in the last section.

\section{Quadratic auxiliary variable (QAV) approach}

In this section, we present the QAV approach  to reformulate the KdV equation into an equivalent form, which transforms the original energy conservation law into two quadratic invariants of the new system. It is worth noting that the original energy reduces a weak invariant of the new system with the consistent initial condition. To our surprise, the QAV reformulation will provide an elegant platform, which allows a class of RK methods to be used directly to develop arbitrarily high-order energy-preserving algorithms that conserve the original energy exactly.

From a mathematical point of view, the KdV equation \eqref{KdVeq} has a bi-Hamiltonian structure, since it can be written in Hamiltonian form in two different ways \cite{Karasozen2013Energy}. We here consider the following energy-preserving Hamiltonian formulation
\begin{equation}\label{Hamilton-KdV}
	u_t = \mathcal{D} \frac{\delta \mathcal{H}}{\delta u},
\end{equation}
where $\mathcal{D} = \partial_x$ and 
$
\frac{\delta \mathcal{H}}{\delta u} = -\frac{\eta}{2} u^2   - \mu^2 u_{xx} 
$ is the variational derivative of the Hamiltonian functional
\begin{equation}\label{Hamilton-energy}
	\mathcal{H}[u] = \int_{a}^{b} \left(-\frac{\eta}{6} u^3 + \frac{\mu^2}{2} u_x^2 \right) \mathrm{d} x.
\end{equation}
It is readily shown that the model \eqref{Hamilton-KdV} with the periodic boundary condition satisfies the following energy conservation law
\begin{equation}
	\frac{\mathrm{d}\mathcal{H}}{\mathrm{d}t} = \left(\frac{\delta \mathcal{H}}{\delta u},u_t\right) = \left(\frac{\delta \mathcal{H}}{\delta u},\mathcal{D} \frac{\delta \mathcal{H}}{\delta u}\right) = 0,
\end{equation}
which implies 
\begin{equation}\label{ECL}
	\mathcal{H}(t) \equiv \mathcal{H}(0),
\end{equation} 
where $(f,g) = \int_{a}^{b} f g \mathrm{d} x$ and the associated $L^2$ norm $\|f\| = \sqrt{(f,f)}$ for any $ f, g\in L^2([a,b]). $ In addition, the KdV equation possesses  mass conservation law
\begin{equation}
	\frac{\mathrm{d}}{\mathrm{d}t} (u,1) = (u_t,1) = \left(\partial_x \frac{\delta \mathcal{H}}{\delta u},1\right) = 0.
\end{equation}
This means
\begin{equation}\label{MCL}
	(u, 1) \equiv \big(u_0(x), 1\big).
\end{equation}
The conservative properties \eqref{ECL} and \eqref{MCL} are important for the correct numerical simulation of such problem.

Next we propose the QAV approach to reformulate the KdV equation \eqref{Hamilton-KdV}. Introducing a quadratic auxiliary variable 
\begin{equation}\label{auxiliary-variable}
	q = u^2,
\end{equation}
the original energy \eqref{Hamilton-energy} can be written into a modified quadratic form
\begin{equation}\label{modified-energy}
	\mathcal{E}[u,q] =  -\frac{\eta}{6} (u,q) + \frac{\mu^2}{2} \|u_x\|^2.
\end{equation}
According to energy variational principle, we reformulate the model \eqref{Hamilton-KdV} to an equivalent system
\begin{equation}\label{QAVsystem}
	\begin{cases}
		u_t = \partial_x \left(-\frac{\eta}{6} q - \frac{\eta}{3} u^2 - \mu^2 u_{xx}\right),\\
		q_t = 2u u_t,\\
	\end{cases}
\end{equation}
with the consistent initial condition
\begin{equation}\label{CIC}
	q(x,0) = \big(u(x,0)\big)^2.
\end{equation}
Letting $z = (u, q)^T$,  the system \eqref{QAVsystem} can be written in the following Poisson form
\begin{equation}
	z_t = \mathcal{B}(z) \frac{\delta \mathcal{E}}{\delta z},
\end{equation}
where the modified energy $\mathcal{E}$ is defined in \eqref{modified-energy} and the skew-adjoint operator $\mathcal{B}(z)$ is given by
\begin{equation*}
	\mathcal{B}(z) = \left(\begin{array}{ll}
		\partial_x & 2 \partial_x u\\
		2u \partial_x & 4u \partial_x u
	\end{array}\right).
\end{equation*}

\begin{thm}
	Under periodic boundary conditions, the QAV system \eqref{QAVsystem} satisfies the following conservation laws
	\begin{align}
		& \big(u(x,t),1\big) \equiv \big(u(x,0),1\big),\quad \forall\; t,\label{QAVMCL}\\
		& q(x,t)-\big(u(x,t)\big)^2 \equiv q(x,0)-\big(u(x,0)\big)^2,\quad \forall\;  x,\; t, \label{QAVCL}\\
		& \mathcal{E}(t) \equiv \mathcal{E}(0), \quad \forall\; t. \label{QAVECL}
	\end{align}
\end{thm}

\begin{proof}
	As described in \eqref{MCL}, we can obtain the mass conservation \eqref{QAVMCL} from the new system \eqref{QAVsystem}. The second equation of the QAV system \eqref{QAVsystem} can be written as
	\begin{equation*}
		\partial_t (q - u^2) = 0,
	\end{equation*}
	which implies \eqref{QAVCL}.  
	
	By some calculations, we can deduce the modified energy conservation law from the QAV system \eqref{QAVsystem}
	\begin{align*}
		\frac{\mathrm{d}\mathcal{E}}{\mathrm{d}t} &= -\frac{\eta}{6} (u_t,q) -\frac{\eta}{6} (u,q_t) - \mu^2 (u_{xx},u_t) \\
		&= \left(-\frac{\eta}{6}q - \frac{\eta}{3} u^2 - \mu^2 u_{xx},u_t\right) \\
		&= \left(-\frac{\eta}{6}q - \frac{\eta}{3} u^2 - \mu^2 u_{xx},\partial_x\left( -\frac{\eta}{6}q - \frac{\eta}{3} u^2 - \mu^2 u_{xx} \right)\right) \\
		&= 0,
	\end{align*}
	which leads to \eqref{QAVECL} and completes the proof.
\end{proof}

\begin{thm}
	Under the consistent initial condition \eqref{CIC}, the QAV system \eqref{QAVsystem} is equivalent to the original model \eqref{Hamilton-KdV}, which conserves the original energy conservation law
	\begin{equation}
		\mathcal{H}(t) \equiv \mathcal{H}(0), \quad \forall\; t. \label{OECL}
	\end{equation} 
\end{thm}

\begin{proof}
	Combining the consistent initial condition \eqref{CIC} with the conservative property \eqref{QAVCL} leads to the auxiliary variable relation \eqref{auxiliary-variable}, which implies that the QAV system \eqref{QAVsystem} is equivalent to the original model \eqref{Hamilton-KdV} and thus \eqref{OECL} holds. This completes the proof.
\end{proof}

\begin{rmk}
	It is worth noting that the QAV reformulated system possesses two quadratic strong invariants, namely, $q-u^2$ and the modified energy $\mathcal{E}.$ Under the consistent initial condition \eqref{CIC}, the QAV system satisfies the auxiliary variable relation \eqref{auxiliary-variable} and the original energy conservation law \eqref{OECL}. Therefore, for the new system, the auxiliary variable relation can be regarded as a weak property, while the original energy reduces a weak invariant \cite{2006Geometric}.
\end{rmk}

\begin{rmk}
	Similar to the IEQ or SAV approaches, the reformulated system satisfies a modified quadratic energy law. But in addition, our QAV reformulation must satisfy that the introduced variable is a quadratic function. In the next section, we will show that the QAV reformulation will provide an elegant platform for developing arbitrarily high-order structure-preserving algorithms that conserve the original energy exactly.
\end{rmk}

\begin{rmk}
	Note that the choice of the quadratic auxiliary variable is not unique. For some complex energy functionals, especially the case of high degree polynomials, we may need to introduce more quadratic auxiliary variables, which will be further studied in our future work. 
\end{rmk}

\section{High-order energy-preserving schemes based on the QAV reformulation}
Traditionally, it is challenging to  develop energy-preserving numerical approximation by using RK method directly. In this section, we derive general RK methods in time for the QAV reformulated system \eqref{QAVsystem}. Among them, a class of RK methods that satisfies the symplectic condition is rigorously proved to preserve the original energy conservation law exactly. Thus we propose a new class of high-order energy-preserving RK methods for the KdV equation based on its QAV reformulation.

Let $\Delta t$ be a time step size and set $t_n = n \Delta t$ for $0 \leq n \leq N_t$ with $T = N_t \Delta t$.  Let $u^n$ denote the numerical approximation to $u(\cdot, t)$ at $t = t_n$ for any function $u$. We now apply an $s$-stage RK method to the QAV system \eqref{QAVsystem}, then the QAV-RK scheme reads
\begin{scm}[$s$-stage QAV-RK Method] \label{scheme:QAV-RK}
	Let $a_{ij}$, $b_i$, $c_i$ $(i,j = 1,\cdots,s)$ be a set of RK coefficients. For given $(u^n, q^n)$, the following intermediate values are first calculated by
	\begin{equation}\label{TDIV}
		\begin{cases}
			U_i =  u^n +  \Delta t \sum\limits_{j=1}^s a_{ij} k_j, \\
			Q_i = q^n +  \Delta t \sum\limits_{j=1}^s a_{ij} l_j, \\
			k_i =  \partial_x \left(-\frac{\eta}{6} Q_i - \frac{\eta}{3} U_i^2 - \mu^2 \partial_{xx} U_i \right), \\
			l_i = 2 U_i k_i.
		\end{cases}
	\end{equation}
	Then $(u^{n+1}, q^{n+1})$ is updated via
	\begin{align}
		u^{n+1} = u^n +  \Delta t \sum\limits_{i=1}^s b_i k_i ,\label{TD-u} \\
		q^{n+1} = q^n +  \Delta t \sum\limits_{i=1}^s b_i l_i.\label{TD-q}
	\end{align}
\end{scm}

\noindent Note that \textbf{Scheme \ref{scheme:QAV-RK}} is a time semi-discrete system, where the variables $U_i, Q_i, k_i, l_i, u^{n+1}$ and $q^{n+1}$ are functions of the spatial variable $x$. As we know, all RK methods preserve linear invariants, while those that satisfy the symplectic condition conserve all quadratic invariants \cite{Cooper1987Stability}. Therefore, for general QAV-RK methods, we have the following theorem for the structure-preserving properties.

\begin{thm}\label{thm:SD-CL}
	If the coefficients of a QAV-RK method satisfy the symplectic condition
	\begin{equation}\label{RK-symplectic-condition}
		b_i a_{i j} + b_j a_{j i} = b_i b_j,\quad \forall~i,j = 1,\cdots,s,
	\end{equation}
	then under periodic boundary conditions, it satisfies the following conservative properties
	\begin{align}
		& \big(u^{n+1},1\big) = \big(u^n,1\big),\label{QAVRKMCL}\\
		& q^{n+1}-\big(u^{n+1}\big)^2 = q^n-\big(u^n\big)^2, \label{QAVRKCL}\\
		& \mathcal{E}^{n+1} = \mathcal{E}^n, \label{QAVRKECL}
	\end{align}
	where $\mathcal{E}^n = -\frac{\eta}{6} (u^n,q^n) + \frac{\mu^2}{2} \|\partial_x u^n\|^2.$
\end{thm}

\begin{proof}
	Under periodic boundary conditions, it is readily to obtain from the third equation of the system \eqref{TDIV} that $(k_i,1) = 0.$ Taking the inner product of \eqref{TD-u} with $1$, we have 
	\begin{equation*}
		\big(u^{n+1},1\big) = \big(u^n,1\big) + \Delta t \sum\limits_{i=1}^s b_i (k_i,1) = \big(u^n,1\big).
	\end{equation*}
	
	According to Eq. \eqref{TD-u}, we can derive
	\begin{equation}\label{u2diff1}
		\big(u^{n+1}\big)^2 - \big(u^n\big)^2 = 2\Delta t\sum\limits_{i=1}^s b_i k_i u^n + \Delta t ^2 \sum\limits_{i,j=1}^s b_i b_j k_i  k_j.
	\end{equation}
	Applying $u^n = U_i -  \Delta t \sum\limits_{j=1}^s a_{ij} k_j$ to the right hand side of \eqref{u2diff1}, we deduce
	\begin{equation}\label{u2diff2}
		\big(u^{n+1}\big)^2 - \big(u^n\big)^2 = 2\Delta t\sum\limits_{i=1}^s b_i k_i U_i,
	\end{equation}
	where $\sum\limits_{i,j=1}^s b_i a_{ij} k_i k_j = \sum\limits_{i,j=1}^s b_j a_{ji} k_i k_j$ and the condition \eqref{RK-symplectic-condition} were used. Noticing that 
	\begin{equation*}
		q^{n+1} - q^n = \Delta t \sum\limits_{i=1}^s b_i l_i = 2\Delta t \sum\limits_{i=1}^s b_i U_i k_i,
	\end{equation*}
	thus we obtain
	\begin{equation*}
		\big(u^{n+1}\big)^2 - \big(u^n\big)^2  = q^{n+1} - q^n,
	\end{equation*}
	which implies \eqref{QAVRKCL}.
	
	Similar to Eq. \eqref{u2diff2}, we can deduce from \textbf{Scheme \ref{scheme:QAV-RK}}
	\begin{align}
		& (u^{n+1},q^{n+1}) - (u^n,q^n) = \Delta t\sum\limits_{i=1}^s b_i (k_i,Q_i) + \Delta t\sum\limits_{i=1}^s b_i (l_i,U_i) = \Delta t\sum\limits_{i=1}^s b_i (k_i,Q_i + 2 U_i^2), \label{QAVRKE1}\\
		& \|\partial_x u^{n+1}\|^2 - \|\partial_x u^n\|^2 = 2\Delta t\sum\limits_{i=1}^s b_i (\partial_x k_i,\partial_x U_i) =  -2\Delta t\sum\limits_{i=1}^s b_i ( k_i,\partial_{xx} U_i). \label{QAVRKE2}
	\end{align}
	Multiplying \eqref{QAVRKE1} and \eqref{QAVRKE2} by $-\frac{\eta}{6}$ and $\frac{\mu^2}{2}$, respectively, then adding the results and noticing $k_i =  \partial_x \left(-\frac{\eta}{6} Q_i - \frac{\eta}{3} U_i^2 - \mu^2 \partial_{xx} U_i \right)$, we have
	\begin{equation*}
		\mathcal{E}^{n+1} - \mathcal{E}^n = \Delta t\sum\limits_{i=1}^s b_i (k_i,-\frac{\eta}{6} Q_i - \frac{\eta}{3} U_i^2 - \mu^2 \partial_{xx} U_i) = 0,
	\end{equation*}
	which leads to the modified energy conservation law. The proof is completed.
\end{proof}

\begin{thm}\label{thm:original-energy-conservation}
	Under the consistent initial condition $q^0 = \big(u^0\big)^2$, the QAV-RK method that satisfies the symplectic condition conserves the original energy conservation law
	\begin{equation}\label{QAVRK-OECL}
		\mathcal{H}^{n} \equiv \mathcal{H}^0, \quad \forall\; n,
	\end{equation}
	where $\mathcal{H}^n = -\frac{\eta}{6} \big((u^n)^3,1\big) + \frac{\mu^2}{2} \|\partial_x u^n\|^2.$
\end{thm}

\begin{proof}
	According to  the property \eqref{QAVRKCL} and the consistent initial condition, we obtain
	\begin{equation}
		q^n = \big(u^n\big)^2, \quad \forall\; n,
	\end{equation}
	which implies
	\begin{equation}\label{energy-relationship}
		\mathcal{E}^n = -\frac{\eta}{6} \left(u^n,\big(u^n\big)^2\right) + \frac{\mu^2}{2} \|\partial_x u^n\|^2 = \mathcal{H}^n.
	\end{equation}
	Combining the modified energy conservation law \eqref{QAVRKECL} and Eq. \eqref{energy-relationship} leads to the original energy conservation law \eqref{QAVRK-OECL}. This completes the proof.	
\end{proof}

\begin{rmk}
	According to \textbf{Theorem \ref{thm:original-energy-conservation}}, all QAV-RK methods that satisfy the symplectic condition are to preserve the original energy conservation law. For convenience, this new class of energy-preserving schemes is called QAV-EPRK methods. Due to $q^n = \big(u^n\big)^2,$ the QAV-EPRK method is equivalent to the following scheme.
	\begin{scm}[$s$-stage QAV-EPRK Method] \label{scheme:QAV-EPRK}
		RK coefficients $a_{ij}$, $b_i$, $c_i$ $(i,j = 1,\cdots,s)$ satisfy the symplectic condition \eqref{RK-symplectic-condition}. For given $u^n$, the following intermediate values are first calculated by
		\begin{equation}\label{TDIV-QAV-EPRK}
			\begin{cases}
				U_i =  u^n +  \Delta t \sum\limits_{j=1}^s a_{ij} k_j, \\
				Q_i = \big(u^n\big)^2 +  2\Delta t \sum\limits_{j=1}^s a_{ij} U_j k_j, \\
				k_i =  \partial_x \left(-\frac{\eta}{6} Q_i - \frac{\eta}{3} U_i^2 - \mu^2 \partial_{xx} U_i \right).
			\end{cases}
		\end{equation}
		Then $u^{n+1}$ is updated via
		\begin{equation}
			u^{n+1} = u^n +  \Delta t \sum\limits_{i=1}^s b_i k_i.
		\end{equation}
	\end{scm}
\noindent Different from IEQ and SAV approaches \cite{Yang2017Numerical,Shen2018The,Jiang2020Arbitrarily,Zhang2020Novel,Jiang2021Explicit}, our QAV-EPRK methods can eliminate auxiliary variables and keep the original energy conservation structure. Different from the existing high-order energy-preserving algorithms \cite{Quispel2008A,Brugnano2010Hamiltonian,Hairer2010Energy,Tang2012Time,Wang2012A,Li2016A}, the QAV-EPRK methods do not involve integrals and allow a class of RK methods to be used directly to produce energy-preserving algorithms. Furthermore, the proposed numerical strategy in this paper is also applicable to some non-polynomial cases, which will be further discussed in our future work.
\end{rmk}

\begin{rmk}
	It is well known that the Gaussian collocation methods satisfy the symplectic condition \eqref{RK-symplectic-condition}, so they can be used to develop arbitrarily high-order energy-preserving algorithms based on our theory. Specially, for the $1$-stage Gaussian collocation method, we can deduce the corresponding QAV-EPRK system
	\begin{equation}\label{1sQAVRK}
		\begin{cases}
			U_1 =  u^n +  \frac{\Delta t}{2}  k_1, \\
			Q_1 = \big(u^n\big)^2 +  \Delta t U_1 k_1, \\
			k_1 =  \partial_x \left(-\frac{\eta}{6} Q_1 - \frac{\eta}{3} U_1^2 - \mu^2 \partial_{xx} U_1 \right),\\
			u^{n+1} = u^n +  \Delta t k_1.
		\end{cases}
	\end{equation}
	Eliminating $U_1, Q_1, k_1,$ we have 
	\begin{equation}\label{AVF}
		\frac{u^{n+1} - u^n}{\Delta t} = \partial_x \left(-\frac{\eta}{6} \Big( (u^n)^2 + u^n u^{n+1} + (u^{n+1})^2 \Big) - \mu^2 \partial_{xx} \frac{u^n + u^{n+1}}{2} \right).
	\end{equation}
	It is readily to show that the scheme \eqref{AVF} can be also obtained by applying the AVF method for the original model \eqref{Hamilton-KdV} \cite{Quispel2008A,Gong2014Some}. Therefore, for the KdV equation, the AVF method is a special case of our proposed methods. In addition, according to the book \cite{2006Geometric}, our proposed QAV-EPRK schemes based on the Gaussian collocation coefficients are naturally symmetric.
\end{rmk}

\section{Fully discrete energy-preserving schemes}
In this section, we employ the Fourier pseudo-spectral method in space for \textbf{Scheme \ref{scheme:QAV-EPRK}} to arrive at fully discrete QAV-EPRK schemes, which are shown to conserve the corresponding energy conservation law in the fully discrete level.

To make the paper self-explanatory, we briefly introduce the following notations. Let $N$ be a positive even integer. We denote the spatial domain $\Omega = [a,b]$, which is uniformly partitioned with mesh size $h = (b-a)/N$ into
\begin{equation*}
	\Omega_{h} =
	\left\{x_j | x_j = a + j h, ~ j = 0, 1, \cdots, N-1 \right\}.
\end{equation*}
Let $V_{h} = \big\{u|u=\{u_j | u_j = u(x_j), x_j \in \Omega_{h}\} \big\}$ be the space of grid functions on $\Omega_{h}$. Note that an element of $V_{h}$ can be regarded as a vector, and its basic rules of operation are the same as the vector, unless otherwise stated. For any two grid functions $u,v \in V_{h}$, define the discrete inner product and norm as follows
\begin{equation*}
	(u,v)_{h} = h \sum\limits_{j = 0}^{N-1} u_j v_j, \quad \|u\|_{h} = \sqrt{(u,u)_{h}}.
\end{equation*}

As we know, the Fourier pseudo-spectral method has been widely used to develop the spatial structure-preserving discretization. Here we only introduce $D_1$ to denote the first-order Fourier differential matrix and omit the details due to save space. Interested readers are referred to \cite{Chen2001Multi,Gong2014Some} for details. Applying the Fourier pseudo-spectral method to \textbf{Scheme \ref{scheme:QAV-RK}}, we obtain the following fully discrete scheme.

\begin{scm}[Fully Discrete QAV-RK Method] \label{scheme:RD-QAV-RK}
	Let $a_{ij}$, $b_i$, $c_i$ $(i,j = 1,\cdots,s)$ be a set of RK coefficients. For given $u^n, q^n \in V_{h}$, the following intermediate values are first calculated by
	\begin{equation}\label{FDIV}
		\begin{cases}
			U_i =  u^n +  \Delta t \sum\limits_{j=1}^s a_{ij} k_j, \\
			Q_i = q^n +  \Delta t \sum\limits_{j=1}^s a_{ij} l_j, \\
			k_i =  D_1 \left(-\frac{\eta}{6} Q_i - \frac{\eta}{3} U_i^2 - \mu^2 D_1^2 U_i \right), \\
			l_i = 2 U_i k_i,
		\end{cases}
	\end{equation}
	where $U_i, Q_i, k_i, l_i \in V_{h},$ and $U_i^2, U_i k_i \in V_{h}$ represent two vectors with the elements
	\begin{equation*}
		(U_i^2)_j = (U_i)_j^2, \quad (U_i k_i)_j = (U_i)_j (k_i)_j.
	\end{equation*}
	Then $u^{n+1}, q^{n+1} \in V_{h}$ are updated via
	\begin{align}
		u^{n+1} = u^n +  \Delta t \sum\limits_{i=1}^s b_i k_i ,\label{FD-u} \\
		q^{n+1} = q^n +  \Delta t \sum\limits_{i=1}^s b_i l_i.\label{FD-q}
	\end{align}
\end{scm}

Analogous to the semi-discrete scheme, we have the following theorems for the fully discrete scheme.

\begin{thm}\label{thm:FD-CL}
	If RK coefficients satisfy the symplectic condition \eqref{RK-symplectic-condition}, then the fully discrete QAV-RK scheme satisfies the following conservative properties
	\begin{align}
		& \big(u^{n+1},1\big)_h = \big(u^n,1\big)_h,\label{QAVRKMCL-FD}\\
		& q^{n+1}-\big(u^{n+1}\big)^2 = q^n-\big(u^n\big)^2, \label{QAVRKCL-FD}\\
		& \mathcal{E}^{n+1} = \mathcal{E}^n, \label{QAVRKECL-FD}
	\end{align}
	where $\mathcal{E}^n = -\frac{\eta}{6} (u^n,q^n)_h + \frac{\mu^2}{2} \|D_1 u^n\|_h^2.$
\end{thm}

\begin{thm}\label{thm:original-energy-conservation-FD}
	Under the consistent initial condition $q^0 = \big(u^0\big)^2$, the fully discrete QAV-RK method that satisfies the symplectic condition conserves the original energy conservation law
	\begin{equation}\label{QAVRK-OECL-FD}
		\mathcal{H}^{n} \equiv \mathcal{H}^0, \quad \forall\; n,
	\end{equation}
	where $\mathcal{H}^n = -\frac{\eta}{6} \big((u^n)^3,1\big)_h + \frac{\mu^2}{2} \|D_1 u^n\|_h^2.$
\end{thm}

\begin{rmk}
	As the proofs of \textbf{Theorem \ref{thm:FD-CL}} and \textbf{Theorem \ref{thm:original-energy-conservation-FD}} are similar to the semi-discrete counterparts in \textbf{Theorems \ref{thm:SD-CL}} and \textbf{\ref{thm:original-energy-conservation}}, we omit the details. Similarly, noticing that $q^n = \big(u^n\big)^2,$ \textbf{Scheme \ref{scheme:RD-QAV-RK}} with the symplectic condition \eqref{RK-symplectic-condition} can be equivalently written into the following fully discrete QAV-EPRK method.
\end{rmk}
	\begin{scm}[Fully discrete QAV-EPRK Method] \label{scheme:QAV-EPRK-FD}
		RK coefficients $a_{ij}$, $b_i$, $c_i$ $(i,j = 1,\cdots,s)$ satisfy the symplectic condition \eqref{RK-symplectic-condition}. For given $u^n \in V_h$, the following intermediate values are first calculated by
		\begin{equation}\label{TDIV-QAV-EPRK-FD}
			\begin{cases}
				U_i =  u^n +  \Delta t \sum\limits_{j=1}^s a_{ij} k_j, \\
				Q_i = \big(u^n\big)^2 +  2\Delta t \sum\limits_{j=1}^s a_{ij} U_j k_j, \\
				k_i =  D_1 \left(-\frac{\eta}{6} Q_i - \frac{\eta}{3} U_i^2 - \mu^2 D_1^2 U_i \right).
			\end{cases}
		\end{equation}
		Then $u^{n+1}\in V_h$ is updated via
		\begin{equation}
			u^{n+1} = u^n +  \Delta t \sum\limits_{i=1}^s b_i k_i.
		\end{equation}
	\end{scm}

\begin{rmk} 
The theoretical analysis of high-order schemes is not trivial, especially for arbitrarily high-order schemes based on RK method. It is a very interesting and complicated project that deserves to be investigated in the future work.
In literature, there are a few results on the RK convergence analysis for the nonlinear Schr\"{o}dinger equation \cite{Feng2021SINU} and gradient flow models \cite{Akrivis2019Energy}. 
We emphasize that one can seek their analytical techniques to analyze our proposed schemes, which will be studied in a sequel.
\end{rmk}

\section{Practically structure-preserving implementation}\label{sec:EIP}
As far as we know, most of structure-preserving algorithms are fully implicit for general conservative systems, which require a nonlinear iteration to solve them. In particular, to maintain the conservative property numerically, the iteration error needs to reach the machine accuracy, which makes the calculation cost extremely expensive. Even so, the error of the conserved quantity in a long time computing is still difficult to stabilize in the machine precision because of the accumulation of machine errors (e.g., see \cite{Gong2014Some,GongCiCP2016An}). In order to improve the computational efficiency of structure-preserving algorithms, we here propose a practically structure-preserving iterative technique, which is inspired by the works \cite{BrugnanoCPC2012A,Cai2020An}.

First of all, we set the initial iteration $k_i^{(0)} = 0$. Let $M >0$ be a given integer. For $m=0$ to $M-1$, we compute $k_i^{(m+1)}$ using
\begin{equation}\label{iter_k}
	\begin{cases}
		U^{(m)}_i =  u^n +  \Delta t \sum\limits_{j=1}^s a_{ij} k^{(m)}_j, \\
		Q^{(m)}_i = \big(u^n\big)^2 +  2\Delta t \sum\limits_{j=1}^s a_{ij} U^{(m)}_j k^{(m)}_j, \\
		U^{(m+1)}_i =  u^n +  \Delta t \sum\limits_{j=1}^s a_{ij} k^{(m+1)}_j, \\
		k^{ (m+1)}_i =  D_1 \left(-\frac{\eta}{6} Q^{ (m)}_i - \frac{\eta}{3} (U^{(m)}_i)^2 - \mu^2 D_1^2 U^{(m+1)}_i \right).
	\end{cases}
\end{equation}
If $\max\limits_{i}\|k_i^{ (m+1)} - k_i^{(m)}\|_{\infty}/\|k_i^{(m+1)}\|_{\infty} < \mathrm{Tol}$, we stop the iteration and set $\wt{k}_i = k_i^{(m+1)}$; otherwise, we set $\wt{k}_i = k_i^{(M)}$. Then $\wt{u}^{n+1}$ is updated via
\begin{align}
	\wt{u}^{n+1} = u^n + \Delta t \sum_{i=1}^{s} b_i \wt{k}_i.
\end{align}
Further, we apply the idea of practically invariants-preserving (EIP) method proposed in \cite{Cai2020An} to update $\wt{u}^{n+1}$, so as to obtain the numerical solution $u^{n+1}$. For the sake of clarity, we briefly describe the modified projection method. Since the proposed QAV-EPRK scheme preserves the discrete mass and energy conservation laws of the KdV equation, the projection solution $\wh{u}^{n+1}$ is computed by
\begin{equation}\label{projection_step}
	\begin{cases}
		\wh{u}^{n+1} = \wt{u}^{n+1}  + \lambda \frac{\delta \mathcal{H}}{\delta u}[\wt{u}^{n+1}] + \nu \mathbf{e},\\ 
		(\wh{u}^{n+1},1)_h = (u^0,1)_h,\\
		\mathcal{H}[\wh{u}^{n+1}] = \mathcal{H}[u^0],
	\end{cases}
\end{equation}
where $\lambda$ and $\nu$ are two Lagrange multipliers, and 
$$\mathbf{e} = (1,1,\cdots,1)^T \in V_h,\quad \mathcal{H}[u] = -\frac{\eta}{6} \big(u^3,1\big)_h + \frac{\mu^2}{2} \|D_1 u\|_h^2,\quad \frac{\delta \mathcal{H}}{\delta u}[u] = -\frac{\eta}{2}u^2 - \mu^2 D_1^2 u.$$
According to the first two equations of the system \eqref{projection_step}, we can deduce
\begin{equation}
	\nu = \dfrac{(u^0,1)_h - (\wt{u}^{n+1},1)_h - \lambda\left(\frac{\delta \mathcal{H}}{\delta u}[\wt{u}^{n+1}],1 \right)_h}{|\Omega|},\quad |\Omega| = b-a.
\end{equation}
Denote 
\begin{equation}
	\phi^{n+1} = \wt{u}^{n+1} + \dfrac{(u^0,1)_h - (\wt{u}^{n+1},1)_h}{|\Omega|}\mathbf{e}, \quad \psi^{n+1} = \frac{\delta \mathcal{H}}{\delta u}[\wt{u}^{n+1}] - \dfrac{\left(\frac{\delta \mathcal{H}}{\delta u}[\wt{u}^{n+1}],1 \right)_h}{|\Omega|}\mathbf{e}.
\end{equation}
Then the system \eqref{projection_step} can be written equivalently into
\begin{equation}\label{modified_projection_step}
	\begin{cases}
		\wh{u}^{n+1} = \phi^{n+1}  + \lambda \psi^{n+1},\\ 
		\mathcal{H}[\wh{u}^{n+1}] = \mathcal{H}[u^0],
	\end{cases}
\end{equation}
where only a nonlinear algebraic equation $\mathcal{H}[\phi^{n+1}  + \lambda \psi^{n+1}] = \mathcal{H}[u^0]$ needs to be solved. Applying the Newton iteration method, we have
\begin{equation}
	\lambda_{k+1} = \lambda_k - \dfrac{\mathcal{H}[\phi^{n+1}  + \lambda_k \psi^{n+1}] - \mathcal{H}[u^0]}{\left( \frac{\delta \mathcal{H}}{\delta u}[\phi^{n+1}  + \lambda_k \psi^{n+1}], \psi^{n+1}\right)_h},
\end{equation}
where the initial iteration is taken as $\lambda_0 = 0.$ According to the idea of the EIP method, we update the numerical solution $u^{n+1}$ by computing only one step Newton iteration
\begin{equation}
	u^{n+1} = \phi^{n+1}  - \dfrac{\mathcal{H}[\phi^{n+1} ] - \mathcal{H}[u^0]}{\left( \frac{\delta \mathcal{H}}{\delta u}[\phi^{n+1}], \psi^{n+1}\right)_h} \psi^{n+1}. 
\end{equation}

\begin{rmk}
	The total cost of solving the practically structure-preserving iteration mainly depends on the system \eqref{iter_k}, which is essentially a system of linear equations with constant coefficients with respect to the unknowns $k_i^{(m+1)}$ and can be solved efficiently by the FFT algorithm \cite{GongSISC2020ARBITRARILY}. For traditional structure-preserving algorithms, the iterative tolerance usually needs to be set as machine precision. However, according to the theoretical analysis of the EIP method \cite{Cai2020An}, we can set an appropriately large iterative tolerance, which can still achieve the effect of practically preserving structure.  By comparisons, it greatly reduces the requirement of traditional structure-preserving algorithms and improves the computational efficiency. In particular, the costless EIP correction also helps us to prevent the accumulation of round-off errors so as to ensure the practically preserving structure in a long time simulation. Numerical experiments in the next section verify the effectiveness and efficiency of the practically structure-preserving iterative strategy.
\end{rmk}

\section{Numerical results}
In this section, we focus on the proposed QAV-EPRK methods with the Gaussian collocation coefficients to conduct several numerical experiments, where the convergence rates are firstly presented to demonstrate the high-order accuracy in time and space of the proposed schemes. Subsequently, some benchmark examples are calculated to verify  energy conservation and effectiveness of the newly proposed schemes.  Unless otherwise stated, the default value of the iterative tolerance is set as $\text{Tol} = 1.0 \times 10^{-14}$ and the number of maximum iterative step is fixed to $M = 100$.

\subsection{Accuracy test} \label{eg:AT}
	We first perform simulations to test the convergence rates of the proposed methods, where the QAV-EPRK scheme with $s$-stage is denoted by QAV-EPRK-$s$. We consider the mode \eqref{KdVeq} with the following analytic solution \cite{Alexander1979Galerkin} 
	\begin{align}
		u(x, t) = 3 c\; \sech^2 \left( \kappa x - \omega t - x_0\right),
	\end{align}
	where $\kappa = \frac{\sqrt{ \eta c}}{2 \mu}$ and $\omega =  c \eta \kappa$. 
	The model parameters are set as $\eta =1$, $\mu = 1$, $c = 1$ and $x_0 = 0$. The initial condition is derived from the exact solution.  The computational domain is taken as $\Omega = [-40, 40]$.

	Due to the space limitation, we only take the QAV-EPRK-3 as an example to test the space accuracy. Meanwhile, we choose time step as $\tau  =10^{-4}$ to prevent the errors in time discretization from contaminating our results.  With grid sizes from $N = 100$ to $300$ by using the increment of $50$, the discrete $L^2$ and $L^{\infty}$ errors are calculated up to the final time $T = 1$. The corresponding results are reported in Figure \ref{fig:test-space}, where we clearly observe the spectral accuracy in space for our newly developed scheme.

		\begin{figure}[H]
		\centering
		\subfigure{
			\includegraphics[width=0.35\textwidth,height=0.25\textwidth]{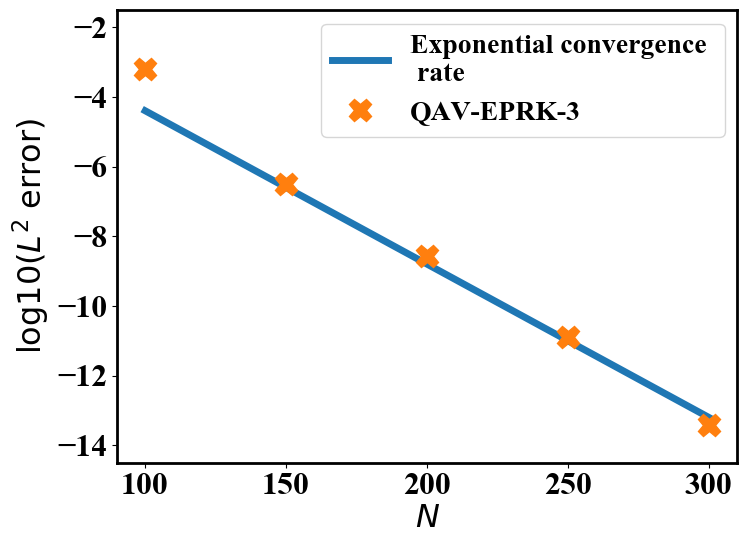}\qquad \qquad
			\includegraphics[width=0.35\textwidth,height=0.25\textwidth]{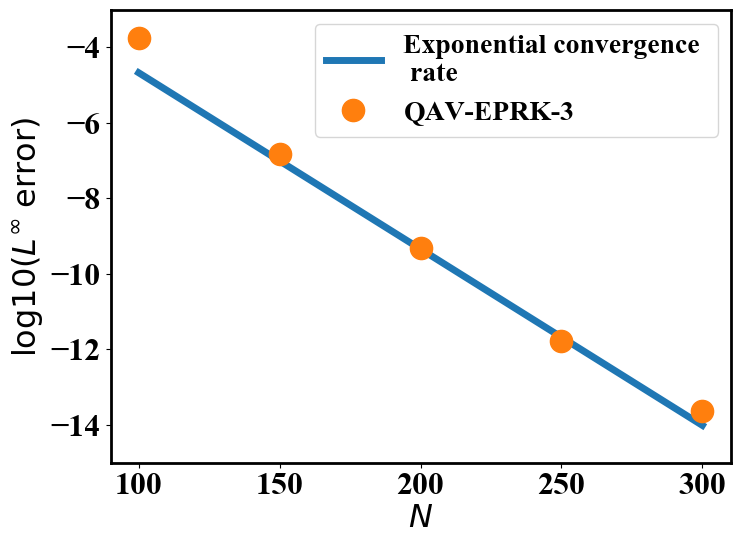}
		}
		\caption{{\bf Subsection} \ref{eg:AT}: The QAV-EPRK-3 scheme serves as an example to show space step refinement test. A spectral accuracy is achieved.\label{fig:test-space}}
	\end{figure}
	
	\begin{figure}[H]
		\centering
		\subfigure{
			\includegraphics[width=0.35\textwidth,height=0.3\textwidth]{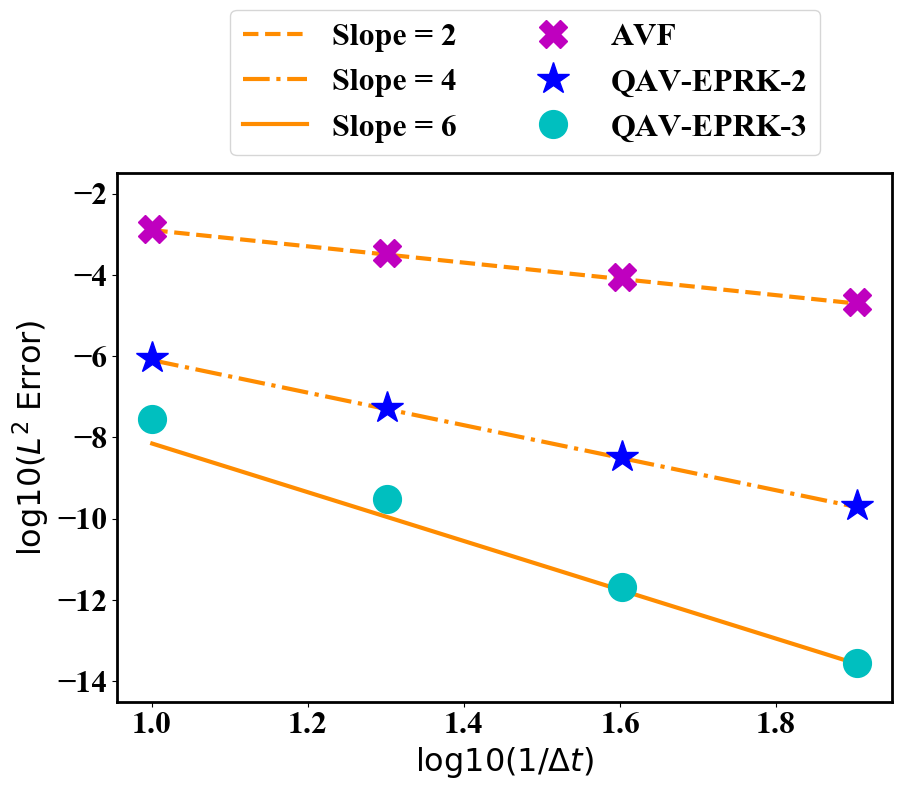}\qquad \qquad
			\includegraphics[width=0.35\textwidth,height=0.3\textwidth]{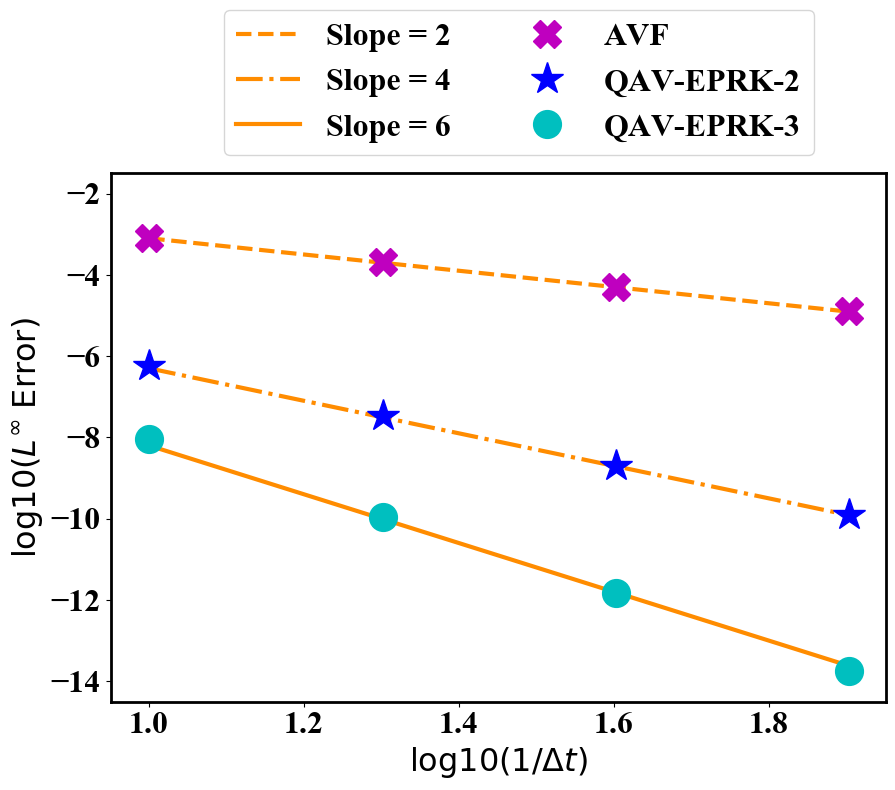}
		}
		\caption{{\bf Subsection} \ref{eg:AT}:  Time step refinement test for the QAV-EPRK schemes. These sub-figures demonstrate the proposed schemes can reach their expected convergence rates. Moreover, their numerical errors for the high-order QAV-EPRK schemes are much smaller than that of the second-order AVF scheme.\label{fig:test-time}}
	\end{figure}

	Next, we test the time convergence rate and choose  $N=512$ spatial meshes. Such a fine mesh can make the spatial discretization error negligible compared with the time discretization error.  In Figure \ref{fig:test-time}, we plot the discrete $L^2$ and $L^{\infty}$ errors at $T = 1$ by varying the time step from $\Delta t = 0.1$ to $\Delta t = 0.0125$ with a factor of $1/2$. We can observe that the two high-order QAV-EPRK schemes exhibit perfect fourth and sixth order accuracy in time as expected, respectively. In particular, the discrete $L^2$ and $L^{\infty}$ errors of the high-order QAV-EPRK schemes are significantly smaller than the second-order AVF scheme with the same time steps.

	Finally, to further demonstrate the advantages of our proposed high-order schemes with the second-order AVF scheme \cite{Quispel2008A}, we test the discrete $L^2$ error for $u$ at $T=10$ smaller than $1.0 \times 10^{-8}$, where the time steps are $\Delta t = 5.0 \times 10^{-5}$ for the second-order AVF scheme, $\Delta t = 10^{-2}$ for the  scheme QAV-EPRK-2 and $\Delta t = 4 \times 10^{-2}$ for the scheme QAV-EPRK-3.
	Their computational costs are summarized in Figure \ref{fig:test-CPU}. It is clearly observed that the high-order QAV-EPRK schemes spend much less CPU time than the second-order AVF scheme to reach the same accuracy, which implies our newly proposed QAV-EPRK schemes are superior to the lower order scheme for accurate in term of long-time simulations.
	\begin{figure}[H]
		\centering
		\subfigure{
			\includegraphics[width=0.50\textwidth,height=0.35\textwidth]{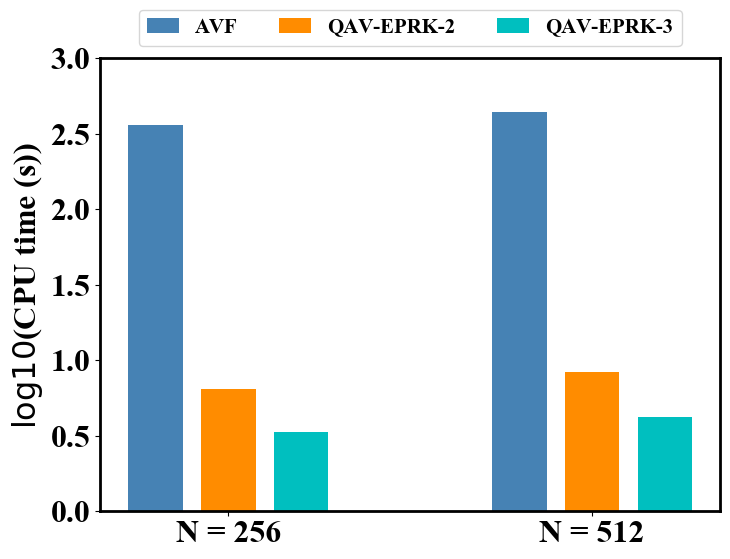}
		}
		\caption{{\bf Subsection} \ref{eg:AT}: Comparison of CPU times between the AVF scheme, QAV-EPRK-2 and QAV-EPRK-3 under the same accuracy. This bar chart shows the high-order schemes perform more superior than the second-order scheme.\label{fig:test-CPU}}
	\end{figure}

\subsection{Invariant test}\label{eg:IT}
\begin{figure}[H]
	\centering
	\subfigure[Comparison of the invariants computed by using QAV-EPRK-2.]{
		\includegraphics[width=0.30\textwidth,height=0.25\textwidth]{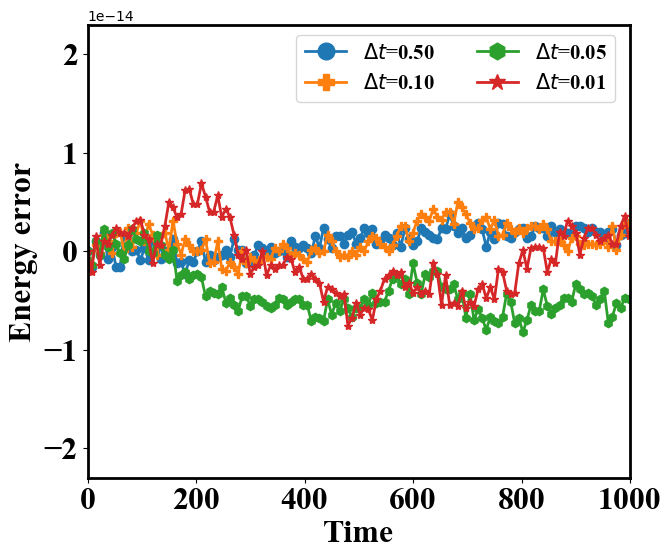}
		\includegraphics[width=0.30\textwidth,height=0.25\textwidth]{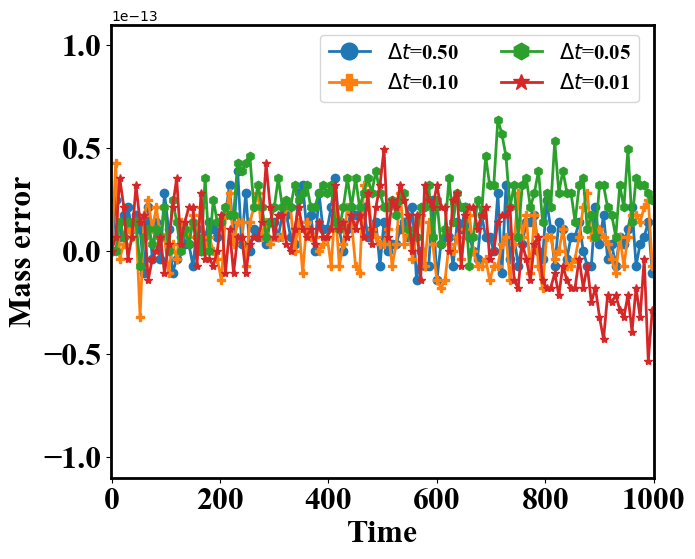}
		\includegraphics[width=0.30\textwidth,height=0.25\textwidth]{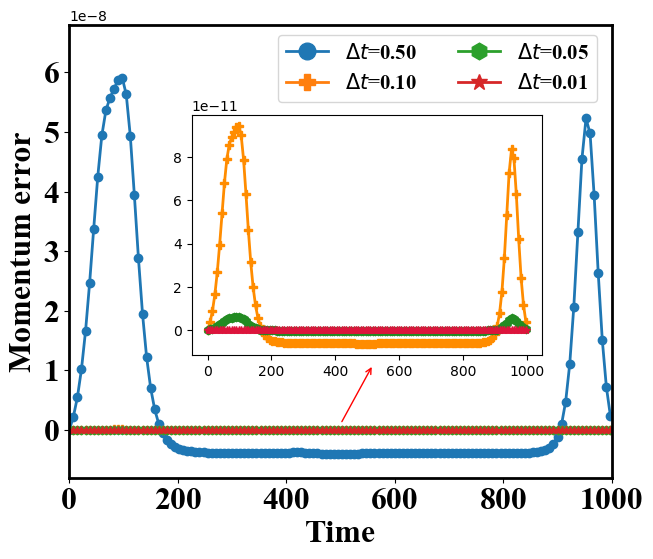}		
	}\\
	\subfigure[Comparison of the invariants computed by using QAV-EPRK-3.]{
		\includegraphics[width=0.30\textwidth,height=0.25\textwidth]{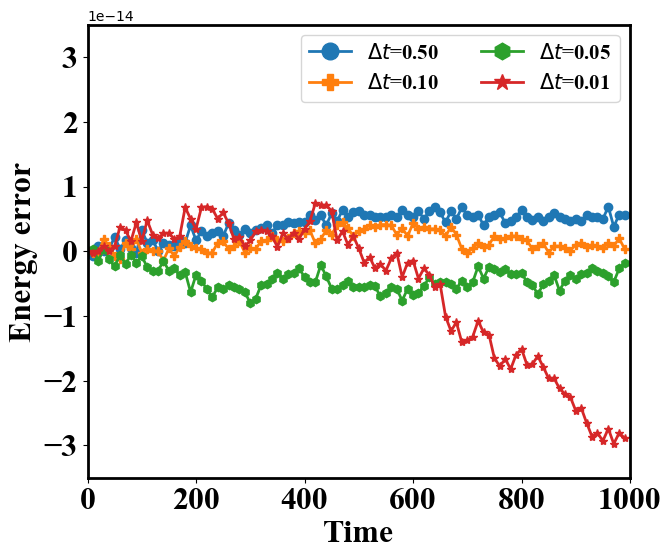}
		\includegraphics[width=0.30\textwidth,height=0.25\textwidth]{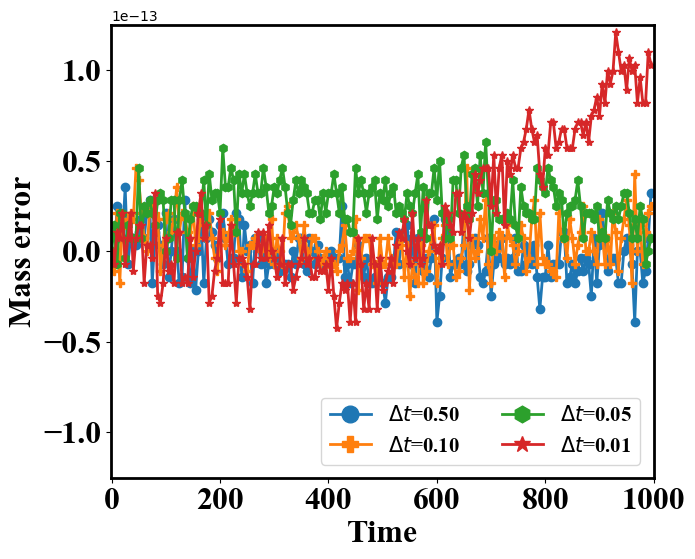}
		\includegraphics[width=0.30\textwidth,height=0.25\textwidth]{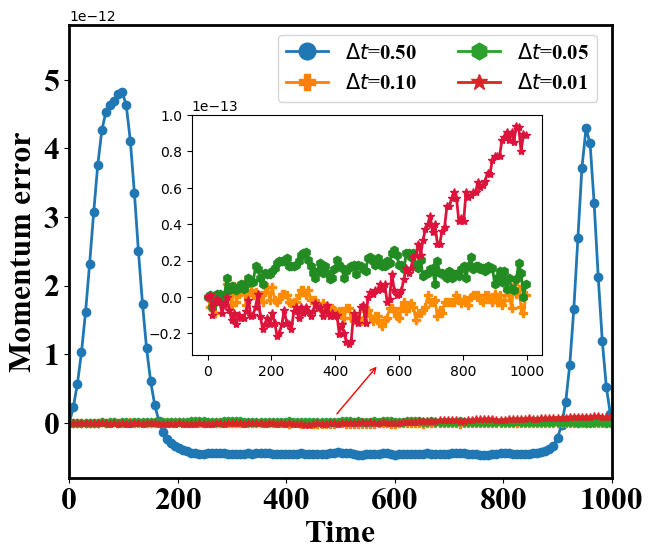}		
	}
	\caption{{\bf Subsection} \ref{eg:IT}: Time evolutions of the three invariants  versus time  calculated by using the high-order QAV-EPRK schemes and different time steps.\label{fig:test-invariants}}
\end{figure}
In this example, we conduct several numerical simulations to test the energy conservation of the developed schemes. We start with the interaction of three solitons  and the corresponding initial condition  is given by
\begin{align}
	u_0(x) = \sum_{i=1}^{3} 12 \kappa_i^2 \sech^2(\kappa_i (x - x_i)), 
\end{align}
with
\begin{align}
	&\kappa_1 = 0.3,\; \kappa_2 = 0.25,\; \kappa_3 = 0.2,\quad
	x_1 = -60,\;  x_2 = -44 ,\;  x_3 = -26.
\end{align}
The model parameters  will be specified as   $\eta = 1$ and $\mu = 1$.  We solve the KdV equation in a periodic domain $\Omega = [-100, 100]$ using a pseudo-spectral method in space with $N = 512$. We carry out different time steps to perform energy conservation. In Figure \ref{fig:test-invariants}, we plot the changes of energy, mass and momentum  computed by using the QAV-EPRK-2 with time steps $\Delta t = 0.5, 0.1$, $0.05$ and $0.01$.  We observe that the errors of the energy are captured accurately and the changes in mass are controlled very well by using the high-order QAV-EPRK schemes. Even though both the QAV-EPRK-2 and QAV-EPRK-3 can not preserve the  momentum conservation,  the errors of momentum that calculated by QAV-EPRK-3 are smaller than that of QAV-EPRK-2.

To further compare the advantages of  our proposed high-order QAV-EPRK schemes with the classic Gauss-type RK (GRK) methods with $s$-stage (abbr. GRK-$s$), we summarized the evolution of energy errors on long-time simulations by using the time step $\Delta t = 0.5$ and the final time $T = 5000$ in Figure \ref{fig:compri-energy}. As expected we see that the  GRK scheme can not preserve the original energy, but the energy error calculated by the GRK scheme with high accuracy is very small. The high-order QAV-EPRK schemes  instead warrant the original energy to machine accuracy. These results strongly support our claim that the technique of QAV provides a new paradigm to develop high-order original-energy-preserving numerical algorithms.

\begin{figure}[H]
	\centering
	\subfigure[Comparison of errors in energy computed by using GRK-2 and QAV-EPRK-2.]{
		\includegraphics[width=0.4\textwidth,height=0.35\textwidth]{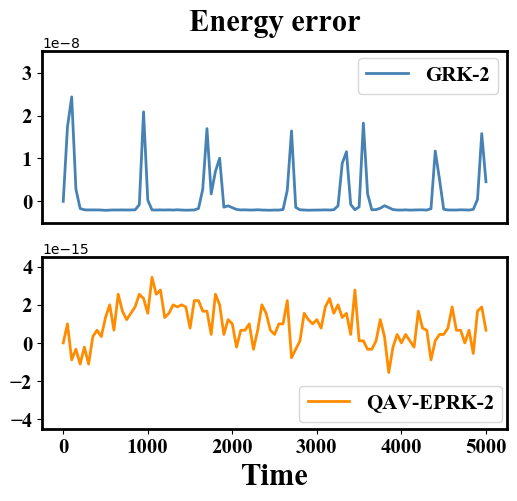}	
	}\qquad
	\subfigure[Comparison of errors in energy computed by using GRK-3 and QAV-EPRK-3.]{
		\includegraphics[width=0.4\textwidth,height=0.35\textwidth]{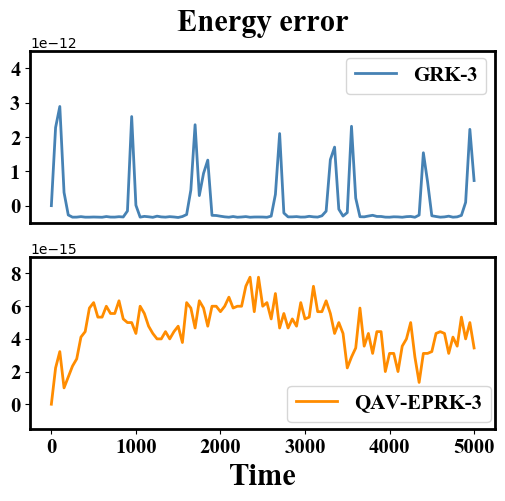}	
	}
	\caption{{\bf Subsection} \ref{eg:IT}: Long-time behavior of energy errors  between the GRK scheme and the QAV-EPRK scheme with time step size $\Delta t = 0.5$.\label{fig:compri-energy}}
\end{figure}

\vspace{-1cm}
\begin{figure}[H]
	\centering
	\subfigure{
		\includegraphics[width=0.30\textwidth,height=0.3\textwidth]{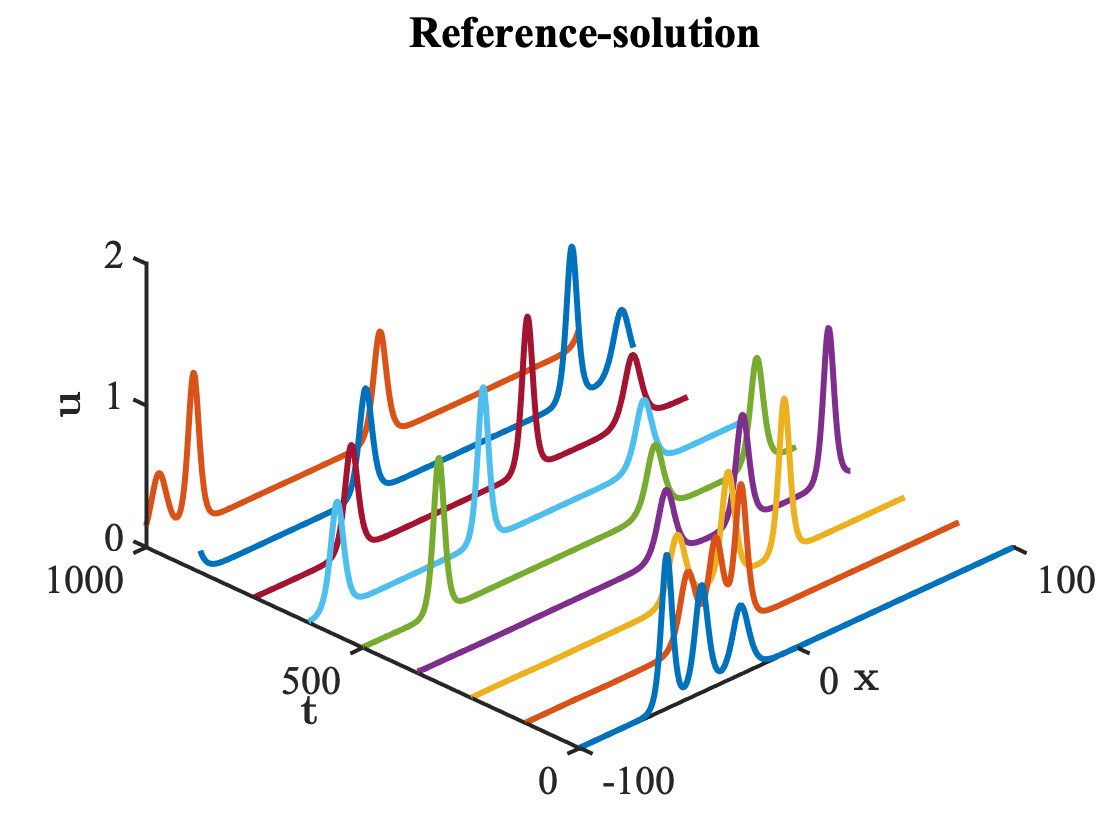}
		\includegraphics[width=0.30\textwidth,height=0.3\textwidth]{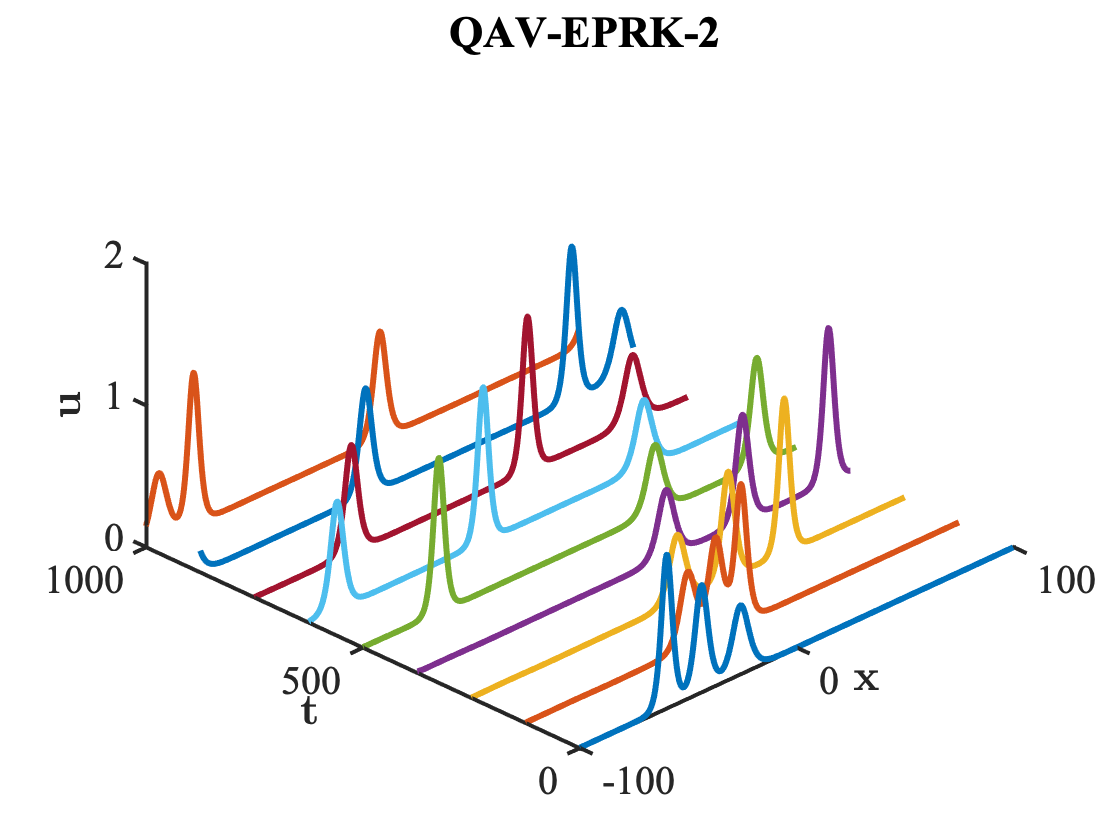}
		\includegraphics[width=0.30\textwidth,height=0.3\textwidth]{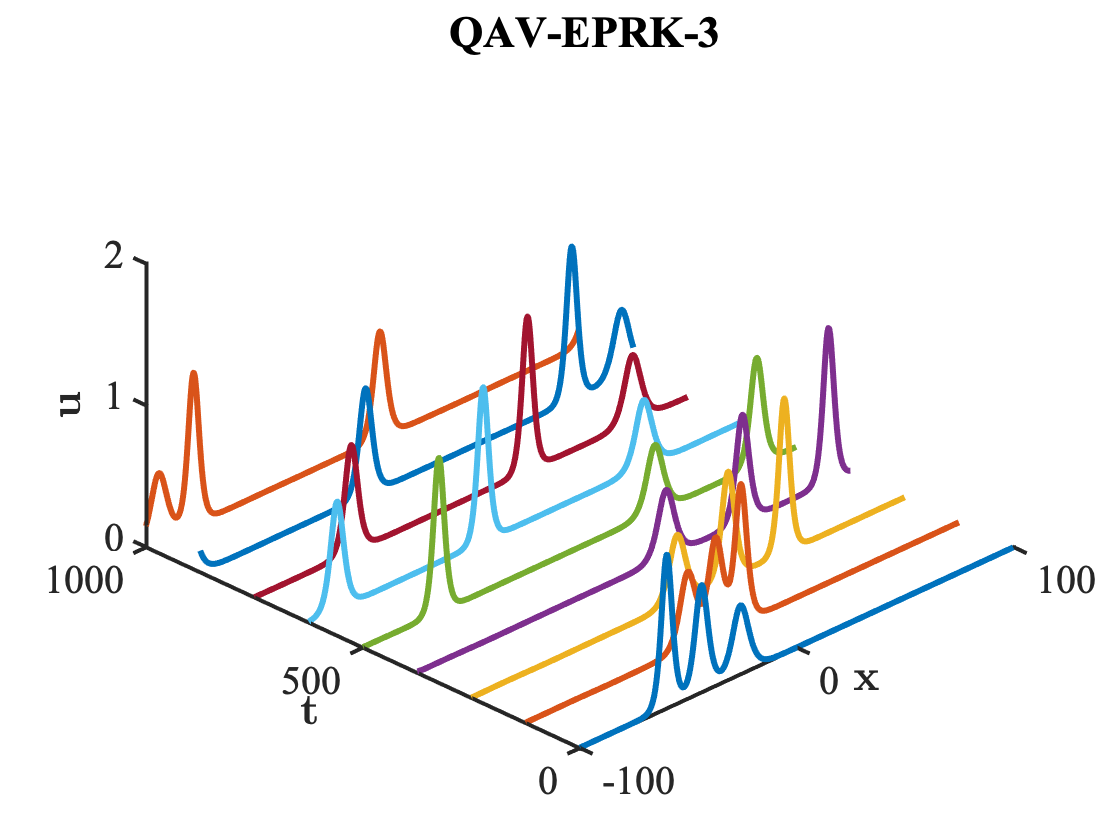}	
	}
	\caption{{\bf Subsection} \ref{eg:IT}: The profiles of numerical solution computed by the high-order QAV-EPRK schemes with time step $\Delta t = 0.1$.\label{fig:three-solitons}}
\end{figure}

As the analytical solution is unknown, we use  the numerical solution from the QAV-EPRK-3 scheme with $\Delta t = 10^{-4}$ as the reference solution. Figure \ref{fig:three-solitons} depicts the profiles of numerical solution  that calculated by the high-order QAV-EPRK schemes and time step $\Delta t = 0.1$ for the motions and interactions of KdV equation with three solitons.  Compared with reference solution, we observe fairly accurate prediction of the motions and interactions of the three solitons with a large time step size  in various time.  These numerical phenomenons are consistent with the reported literatures.
In a word, the  numerical behaviors above support our claim that our proposed high-order schemes are very efficient to deal with the motion and interactions of solitons.

\subsection{Convection-dominant problem}\label{eg:CDP}
\begin{figure}[H]
	\centering
	\subfigure{
		\includegraphics[width=0.4\textwidth,height=0.3\textwidth]{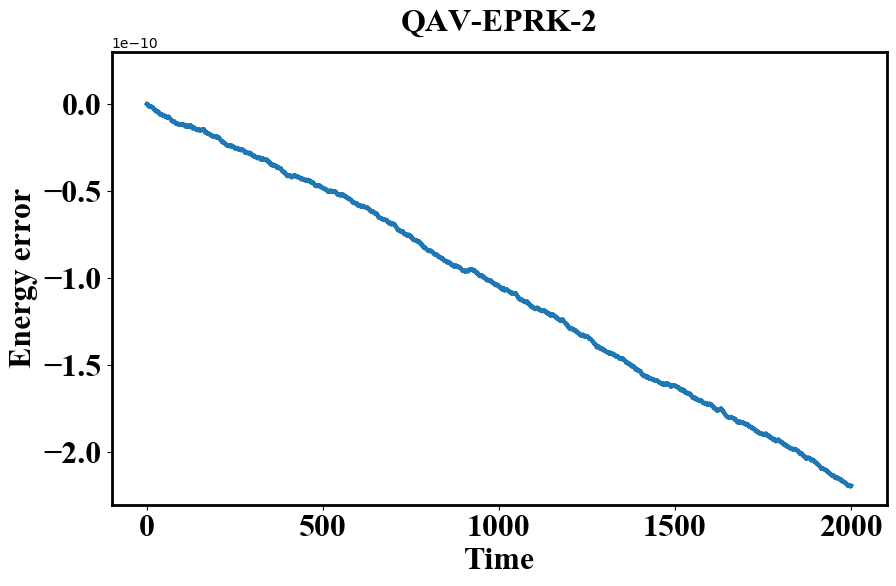}\qquad
		\includegraphics[width=0.4\textwidth,height=0.3\textwidth]{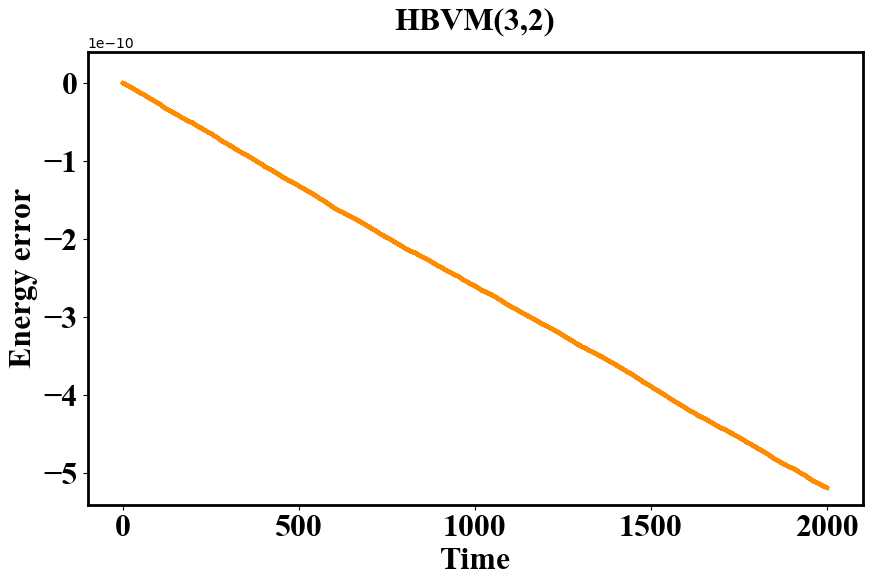}
	}
	\caption{{\bf Subsection} \ref{eg:CDP}: Evolution of errors in original energy produced by QAV-EPRK-2 and  HBVM(3,2).  As one can see that a linear error growth is observed.
		This may be the accumulation of round-off errors in long-time numerical simulation. \label{fig:two-sch-energ}}
\end{figure}
In this example, we test the convection-dominant problem with the interaction  of two solitary waves propagation of the KdV equation ($\eta = 6$, $\mu = 1$), where the initial condition is given by
\begin{align}
	u_0(x) = 12 \frac{3 + 4 \cosh(2 x ) + \cosh(4 x )}{(3 \cosh( x ) + \cosh(3 x ))^2}.
\end{align}
Due to the space limitation, we just take the codes developed from the  QAV-EPRK-2 and HBVM(3,2) (Ref. \cite{Brugnano2010Hamiltonian}) as a demo to simulate this convection-dominant problem in a domain $\Omega = [-20, 20]$ with $N = 256$ spatial meshes.  We perform this simulation with $\Delta t = 0.005$ and depict the errors of the original energy at the end time $T = 2000$. The energy errors for the long-term numerical simulation are listed in Figure \ref{fig:two-sch-energ}. Even though the two schemes theoretically warrant the original energy conservation,  as can be seen in Figure \ref{fig:two-sch-energ} that the amplitude of the errors in the original energy is about $10^{-10}$ but not 
up to machine precision. This reason may be that the use of finite arithmetic may sometimes generate a mild numerical drift of the energy over long-time numerical simulations, which causes the growth of  the  energy error. 
\begin{figure}[H]
	\centering
	\subfigure{
		\includegraphics[width=0.45\textwidth,height=0.3\textwidth]{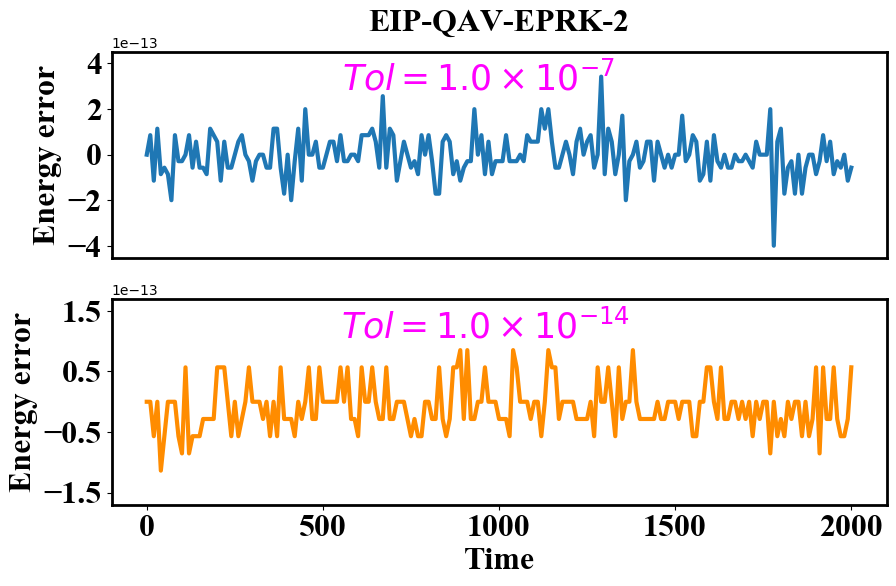}\qquad
		\includegraphics[width=0.45\textwidth,height=0.3\textwidth]{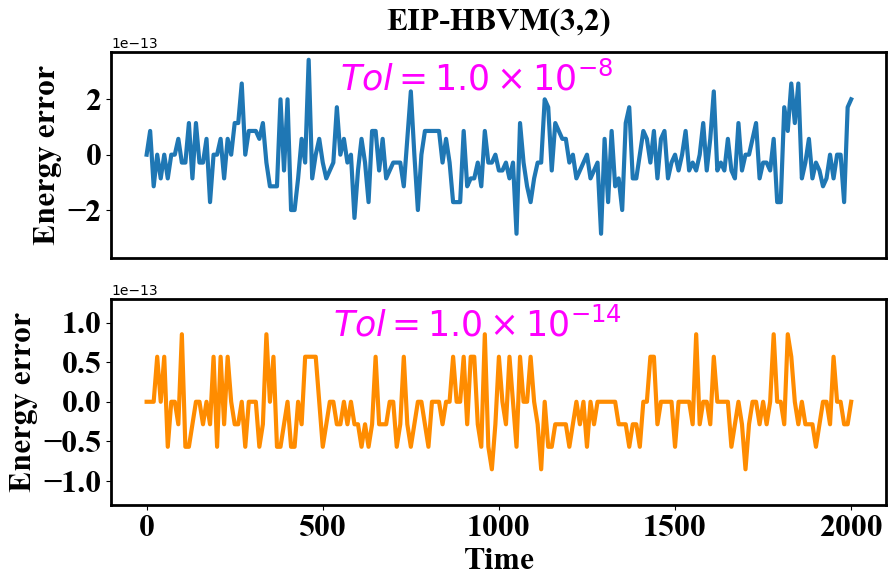}
	}\caption{{\bf Subsection} \ref{eg:CDP}: Evolution of errors in original energy produced by EIP-QAV-EPRK-2 and EIP-HBVM(3,2)  with using various iterative tolerances. \label{fig:two-sch-energ2}}
\end{figure}

\begin{figure}[H]
	\centering
	\subfigure[]{
		\includegraphics[width=0.45\textwidth,height=0.28\textwidth]{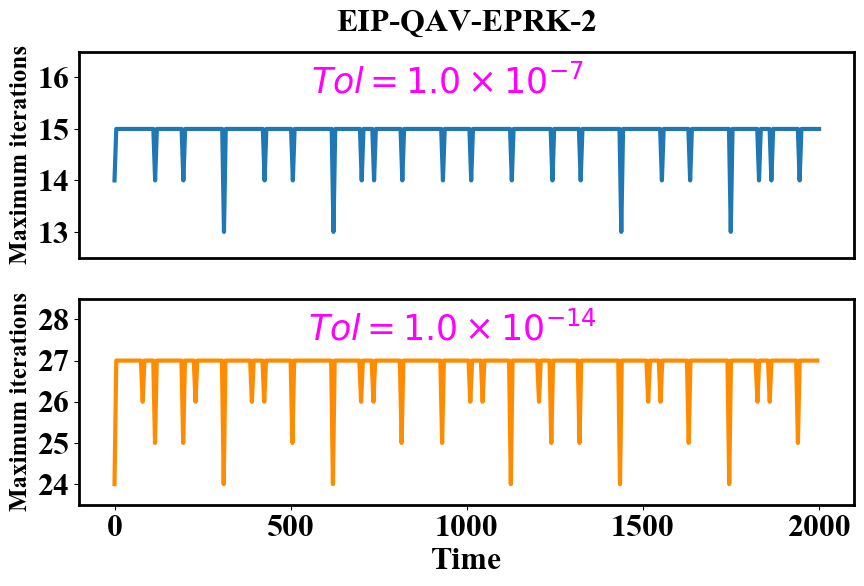}\qquad
		\includegraphics[width=0.45\textwidth,height=0.28\textwidth]{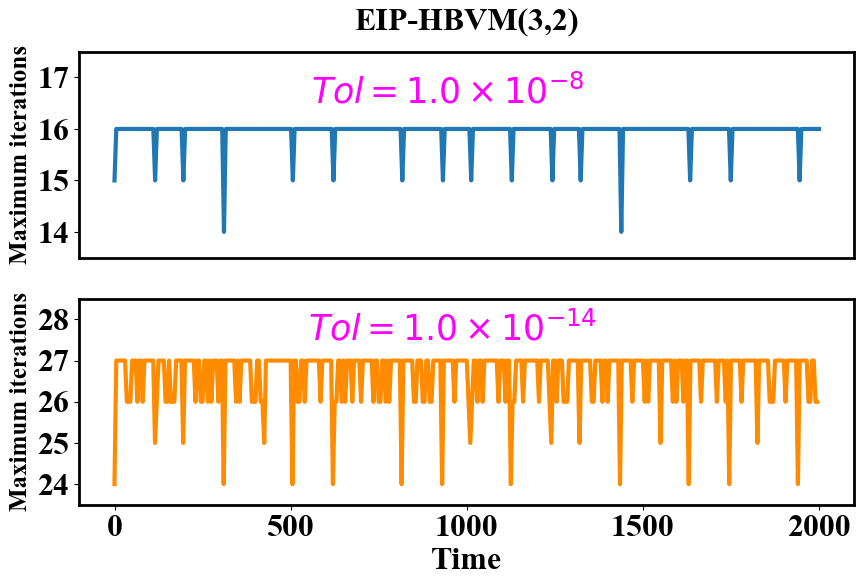}
	}\vspace{-2mm}\\
	\subfigure[]{
		\includegraphics[width=0.5\textwidth,height=0.3\textwidth]{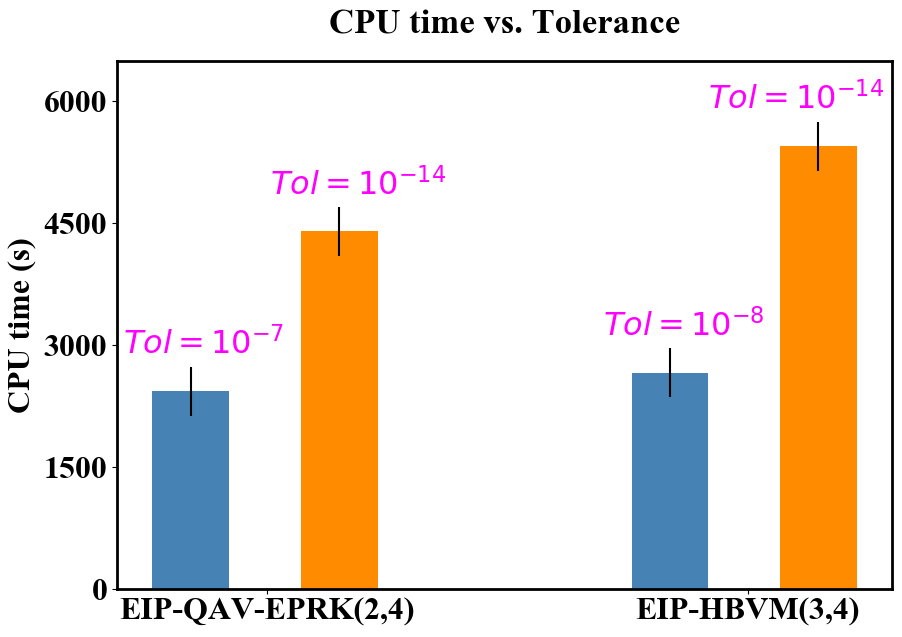}
	}
	\caption{{\bf Subsection} \ref{eg:CDP}: 
		(a) The time evolutions of the maximum iterations by using EIP-QAV-EPRK-2 and EIP-HBVM(3,2)  with different iterative tolerances.
		(b) Comparison of CPU times for the two schemes by using various iterative tolerance from $t=0$ to $t=2000$. This bar chart demonstrates that the practically structure-preserving algorithms improve the computing efficiency using a large tolerance  in long-time numerical simulations when achieving the same numerical behaviors.\label{fig:iter-CPU-compari}}
\end{figure}

\begin{figure}[H]
	\centering
	\subfigure{
		\includegraphics[width=0.4\textwidth,height=0.25\textwidth]{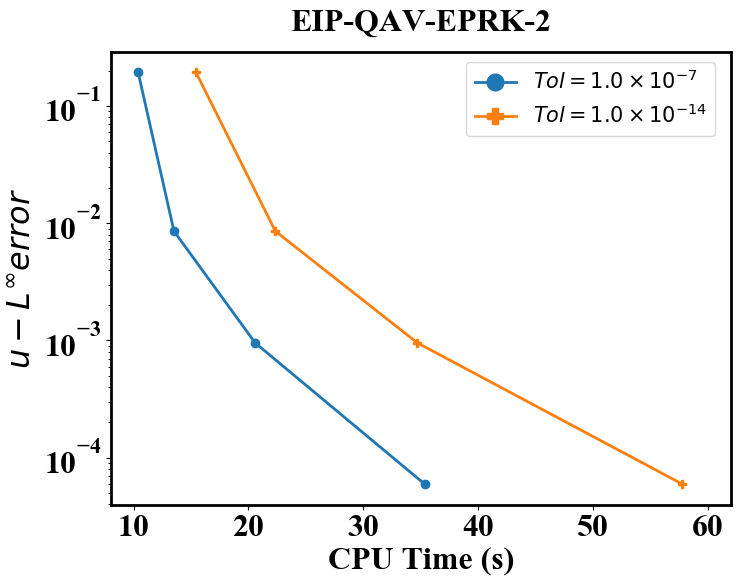}\qquad
		\includegraphics[width=0.4\textwidth,height=0.25\textwidth]{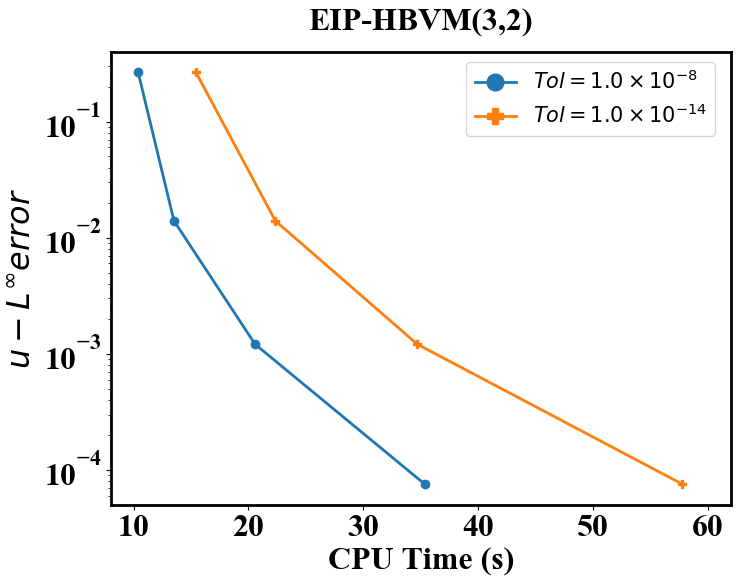}
	}
	\caption{{\bf Subsection} \ref{eg:CDP}:  The maximal error in solution vs. CPU time for various iterative tolerance.\label{fig:Uerror-CPUtime}}
\end{figure}

To circumvent this apparent drawback, we adopt the two schemes that combined with the  EIP technique that described in section \ref{sec:EIP} to perform this example.  For comparison purposes, we set the iterative tolerance $\text{Tol} = 1.0 \times 10^{-7}$ for EIP-QAV-EPRK-2 and  $\text{Tol} = 1.0 \times 10^{-8}$ for EIP-HBVM(3,2) to run these codes and $\text{Tol} = 1.0 \times 10^{-14}$ to calculate numerical solution  as a reference, respectively.  We plot the evolution of the original energy errors in Figure \ref{fig:two-sch-energ2}, the maximum iterations and the total CPU time in Figure \ref{fig:iter-CPU-compari}, respectively.  Compared with the results in Figure \ref{fig:two-sch-energ},  it clearly indicates that the two schemes with EIP technique can easily control the linear growth of the energy errors that generated by the round-off errors, where the original energy errors remain stable and are up to machine precision. 
It follows from Figure \ref{fig:two-sch-energ2} that we clearly observe that the original energy errors that computed by using a large tolerance are consistent with that of the reference tolerance $\text{Tol} = 1.0 \times 10^{-14}$, while Figure \ref{fig:iter-CPU-compari} shows that the two schemes with using large tolerance  greatly save time-consuming and vastly improved the computational efficiency for long-time dynamic simulations when yielding the same numerical effects.
Subsequently, we also the investigate the numerical solution  error in $L^{\infty}$ norm versus the total CPU time for various tolerance at the stopping time $T=10$, where we choose the reference solution that calculated by EIP-QAV-EPRK-3 scheme with $\Delta t = 1.0 \times 10^{-5}$ as the `exact' solution. In Figure \ref{fig:Uerror-CPUtime}, we observe that the practically structure-preserving schemes with a large tolerance can yield the same numerical accuracy as the reference counterpart, but the former is  more effective than the latter in practical calculation.
Thus, these numerical results deeply support our conclusion that our proposed high-order schemes with the EIP technique have a strong practicality in practice. 
Additionally, by comparison, our newly proposed high-order QAV-EPRK schemes can achieve at least the same numerical behaviors as HBVM \cite{Brugnano2019Energy}.
Now, we can draw a conclusion that the high-order QAV-EPRK schemes with EIP technique not only keep the high computational accuracy, but also bring significant computational time saving for solving this model.

\subsection{Random bimodal wave}\label{eg:RBW}
\begin{table}[H]
	{\caption{{\bf Subsection} \ref{eg:RBW}: Parameters of the initial conditions.}\label{Tab:rand-wave}}
	\vspace{-0.2cm}
	\begin{center}
		\begin{tabular}{c  c c c c c c}\hline\specialrule{0em}{2pt}{1pt}
			Case  &  $Q_2/Q_1$    &$k_1$       &$K_1$  &$k_2$      &$K_2$  \\ \hline\specialrule{0em}{1.5pt}{1.5pt}
			I  &   0    &   1  &   0.1    & $--$ &$--$  \\[-1mm]
			II &   0.5 &   1  &   0.1    &0.5&0.05\\[-1mm]
			III &   0.5&    1  &  0.1    &0.5&0.1\\[-1mm]
			IV &      1 &    1  &  0.1    &0.5&0.05\\[-1mm]
			V & 0.5 &     1  &   0.1    &1.5&0.05\\[-1mm]
			VI &     1&      1  &  0.1    &1.5&0.05\\ \specialrule{0em}{2pt}{1pt}\hline
		\end{tabular}
	\end{center}
\end{table}	
\vspace{-2mm}

\begin{figure}[H]
	\centering
	\subfigure{
		\includegraphics[width=0.3\textwidth,height=0.225\textwidth]{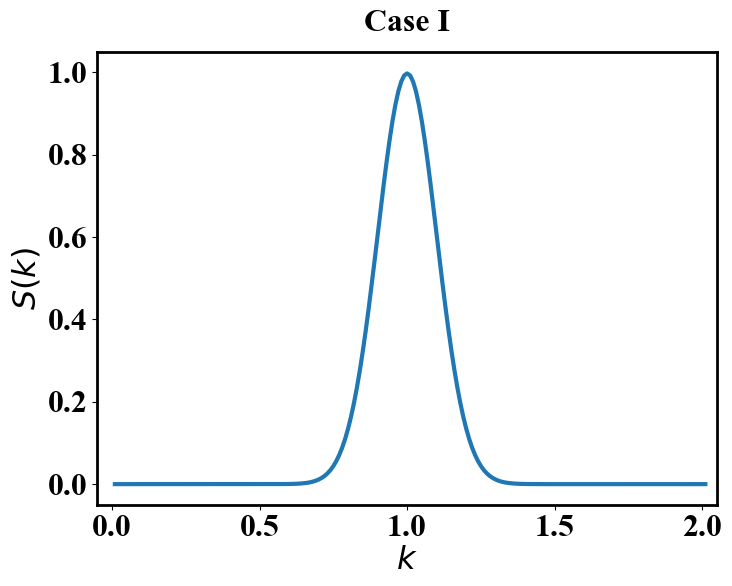}
		\includegraphics[width=0.3\textwidth,height=0.225\textwidth]{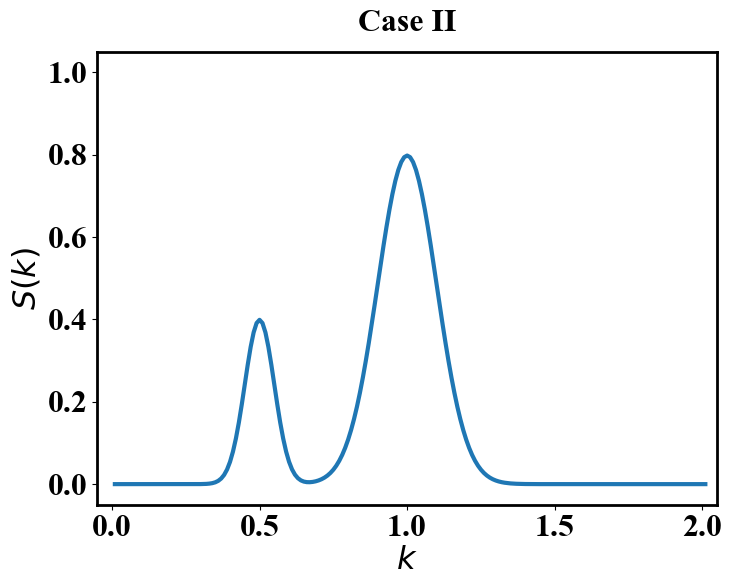}
		\includegraphics[width=0.3\textwidth,height=0.225\textwidth]{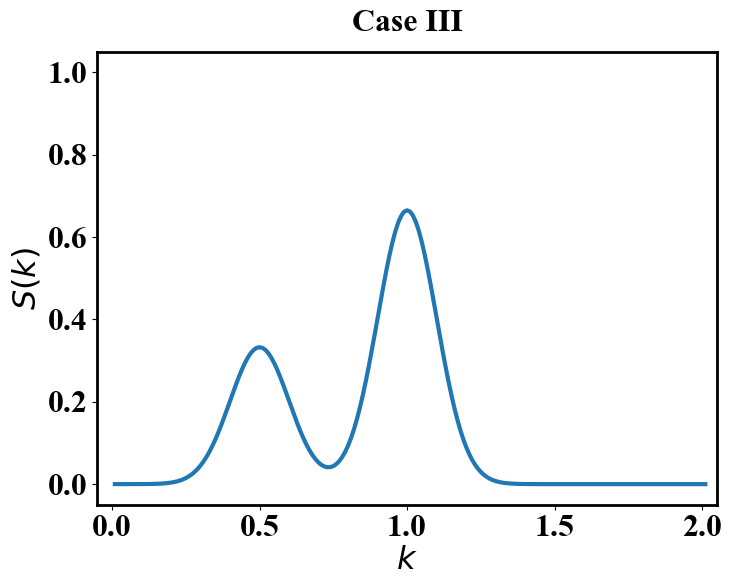}
	}\\
	\subfigure{
		\includegraphics[width=0.3\textwidth,height=0.225\textwidth]{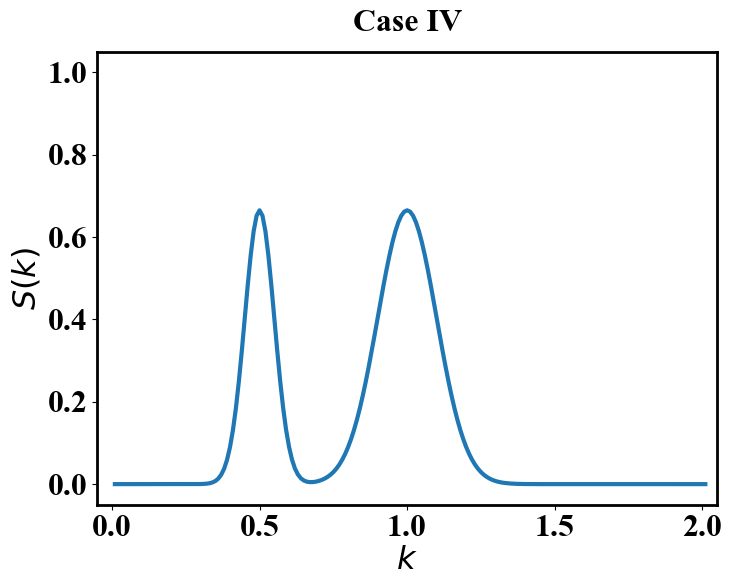}
		\includegraphics[width=0.3\textwidth,height=0.225\textwidth]{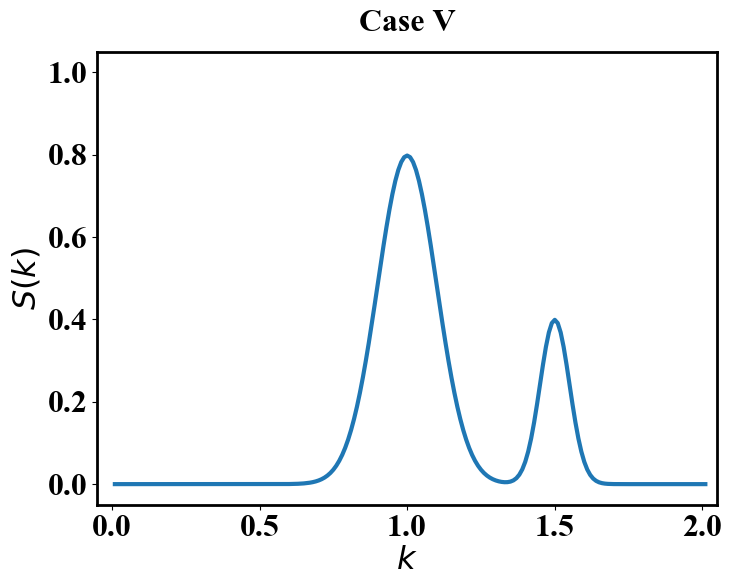}
		\includegraphics[width=0.3\textwidth,height=0.225\textwidth]{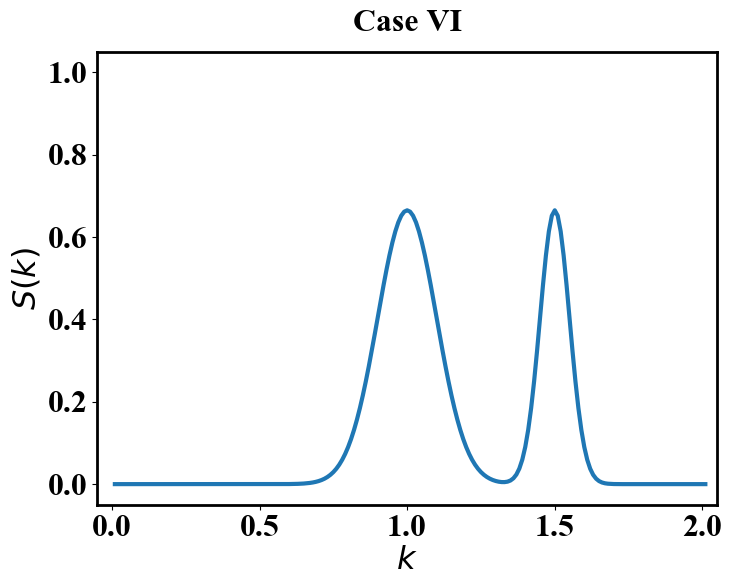}	
	}
	\caption{{\bf Subsection} \ref{eg:RBW}: Shapes of the initial $S(k)$ for six cases, where the parameters are given in Table \ref{Tab:rand-wave}.\label{fig:6cases-SK}}
\end{figure}
In this example, we consider the numerical simulation of random bimodal wave. For more details, interested readers refer to \cite{Did2019Numerical}. The initial datum for the numerical simulation at $t = 0$ is specified in the form of a linear sum of cosines with randomly chosen phases
\begin{align}
	u(x, 0) = \sum_{j = 1}^{\frac{N}{2}-1} \sqrt{2 S(k_j) \Delta k} \cos(k_j x + \psi_j),
\end{align}
where $k_j = j \Delta k$, $j = 1, \cdots N/2 - 1$ are admitted wave-numbers, $\Delta k = 0.01$, $N = 2^{12}$. The model parameters are specified as $\eta = 1$ and $\mu = \sqrt{2/9}$.
Here, the initial phase $\psi_j$ is the random number located in $(0, 2 \pi)$. The coefficients of the wave-numbers power spectrum $S(k_j)$ are given by
\begin{align}
	S(k)  = Q_1 \exp\left(-\frac{(k - k_1)^2}{2 K_1^2} \right) + Q_2 \exp\left(-\frac{(k - k_2)^2}{2 K_2^2} \right),\quad k >0.
\end{align}
In this case, the parameters in this simulation are listed in the following Table \ref{Tab:rand-wave}. In Figure \ref{fig:6cases-SK}, we present the shapes of the initial Fourier transform with a cut off spectrum tail for six cases.

\begin{figure}[H]
	\centering
	\subfigure{
		\includegraphics[width=0.485\textwidth,height=0.25\textwidth]{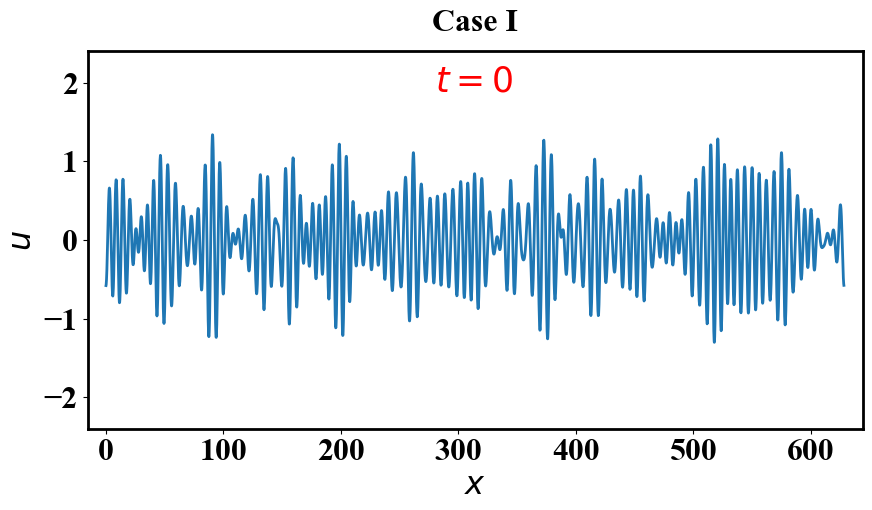}
		\includegraphics[width=0.485\textwidth,height=0.25\textwidth]{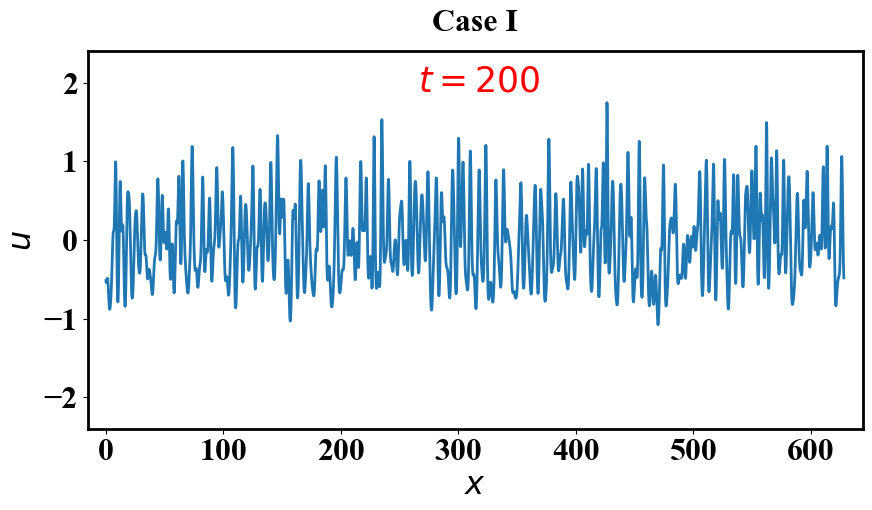}
	}
	\subfigure{
		\includegraphics[width=0.485\textwidth,height=0.25\textwidth]{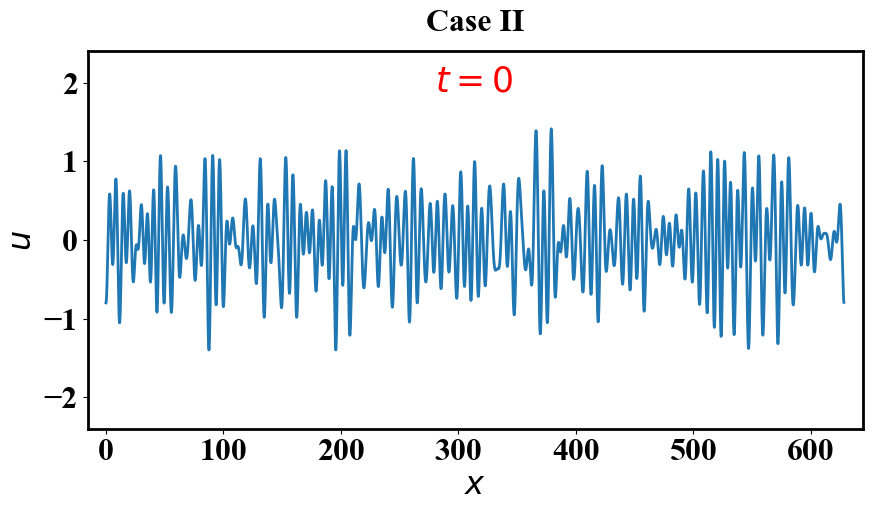}
		\includegraphics[width=0.485\textwidth,height=0.25\textwidth]{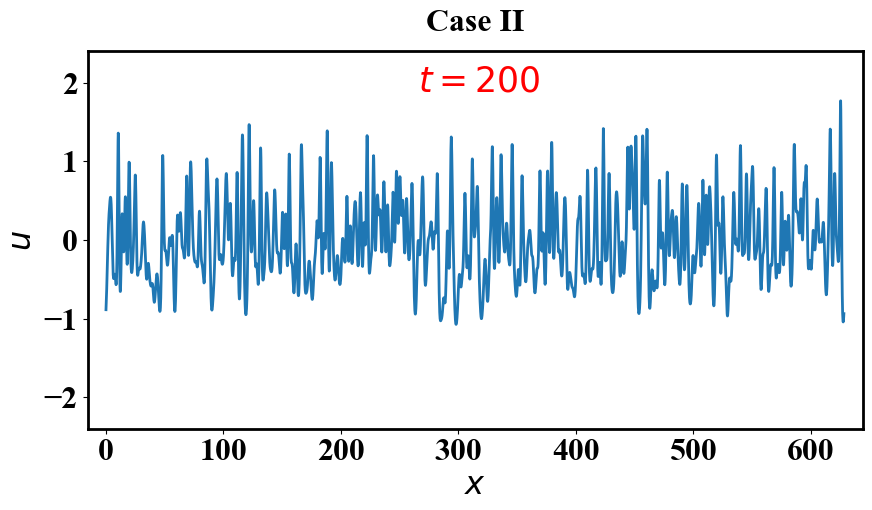}
	}\\
	\subfigure{
		\includegraphics[width=0.485\textwidth,height=0.25\textwidth]{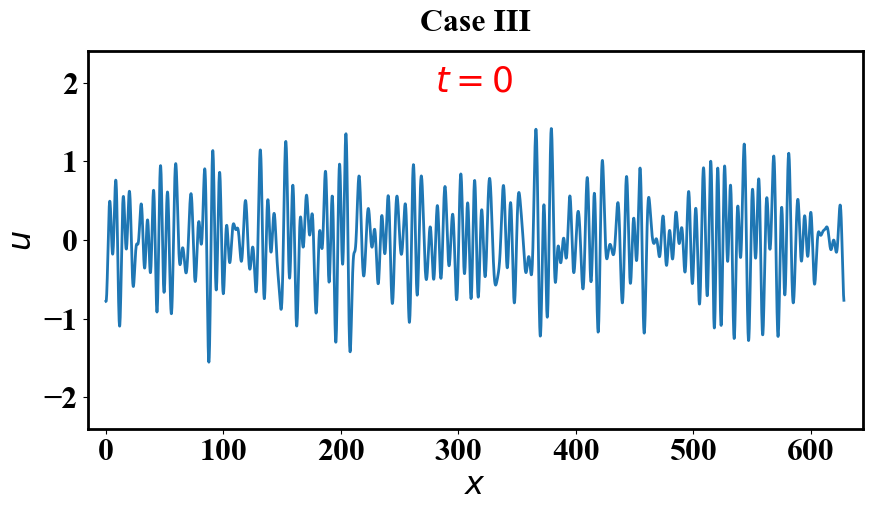}
		\includegraphics[width=0.485\textwidth,height=0.25\textwidth]{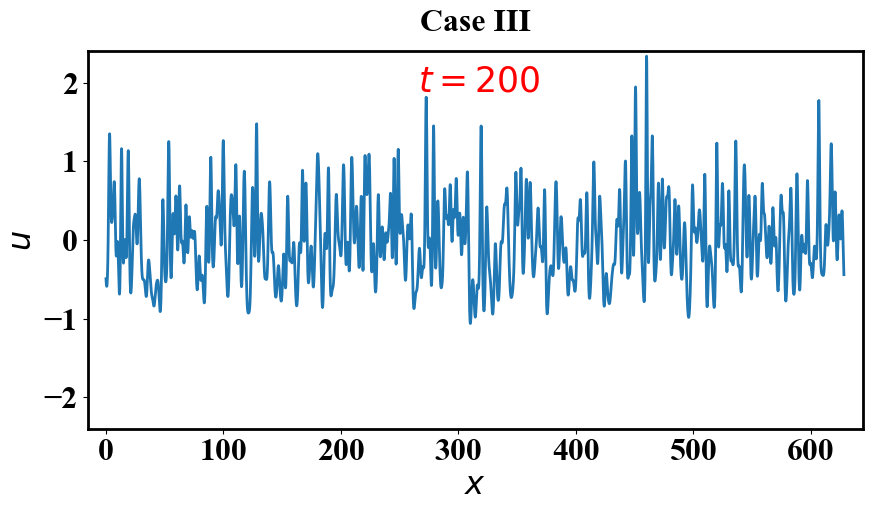}
	}
	\caption{{\bf Subsection} \ref{eg:RBW}: The profiles of numerical solution $u$ that computing by the EIP-QAV-EPRK-2 scheme. The time step is $\Delta t = 0.01$. The snapshots are taken at $t = 0$ and $t = 200$ for Case I, Case II and Case III.\label{fig:123cases-u0}}
\end{figure}

We set  the computed domain $\Omega = [0, 200 \pi]$ and assume the periodic boundary conditions. The space is discretized by using  $N $ Fourier modes and the time step is $\Delta t = 1.0 \times 10^{-2}$.  The previous test has presented the advantages of the high-order QAV-EPRK that combined with the EIP skill. Thus, we will continue to  
adopt the code of EIP-QAV-EPRK scheme to simulate this example.
The profiles of $u$ at the initial time and $t=200$ for six different cases are shown in Figure \ref{fig:123cases-u0} and Figure \ref{fig:456cases-u0}. These numerical phenomenons are consistent  with the numerical solution obtained by using the second-order numerical solver in the literature \cite{Did2019Numerical}, while we can use a relatively larger time step than the time-step used in the other methods. In Figure \ref{fig:123cases-energy} and Figure \ref{fig:456cases-energy}, we plot the energy errors of all cases with time.  As one can see that all curves of energy error are up to machine epsilon and keep stable. This confirms that the original energy is clearly very well preserved for all cases by using the EIP-QAV-EPRK-2 scheme. In a word, these numerical behaviors demonstrate the effectiveness  of our proposed schemes once more.

\begin{figure}[H]
	\centering
	\subfigure{
		\includegraphics[width=0.485\textwidth,height=0.25\textwidth]{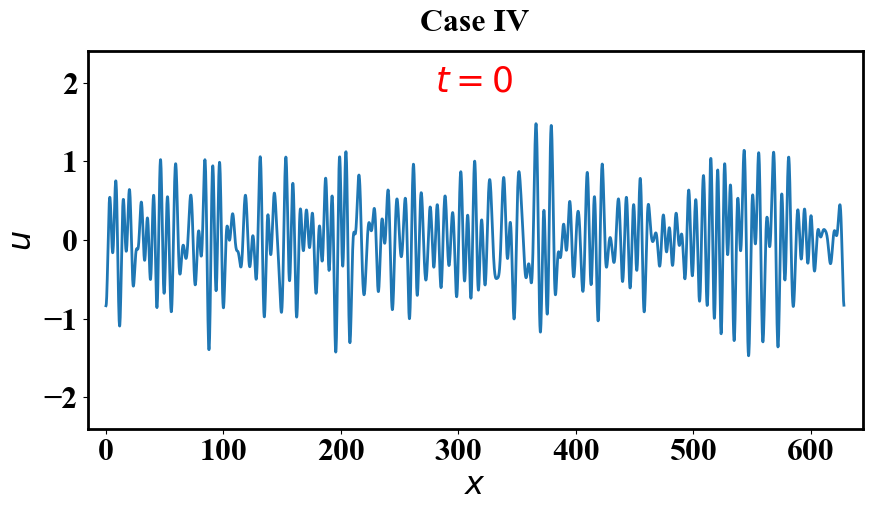}
		\includegraphics[width=0.485\textwidth,height=0.25\textwidth]{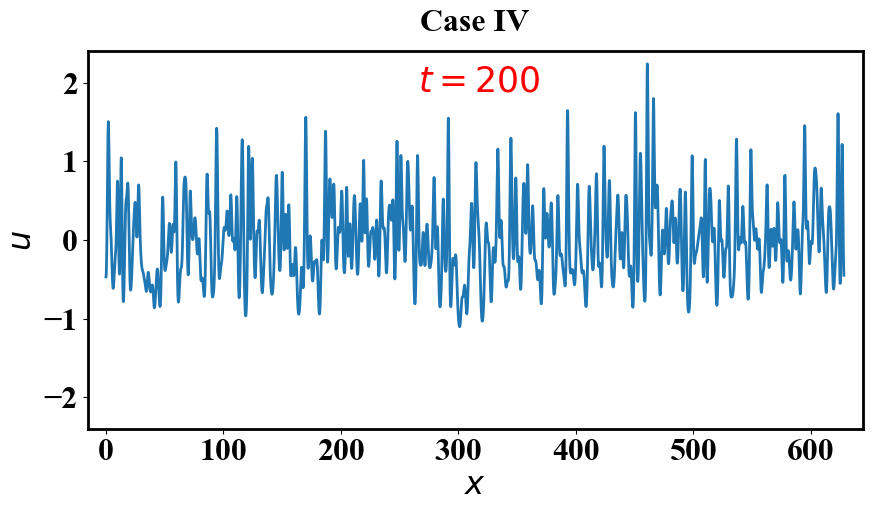}
	}\\
	\subfigure{
		\includegraphics[width=0.485\textwidth,height=0.25\textwidth]{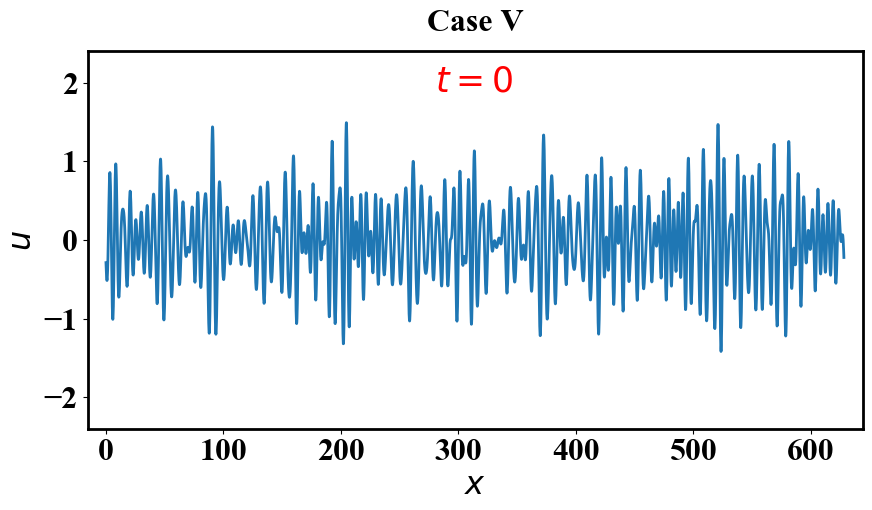}
		\includegraphics[width=0.485\textwidth,height=0.25\textwidth]{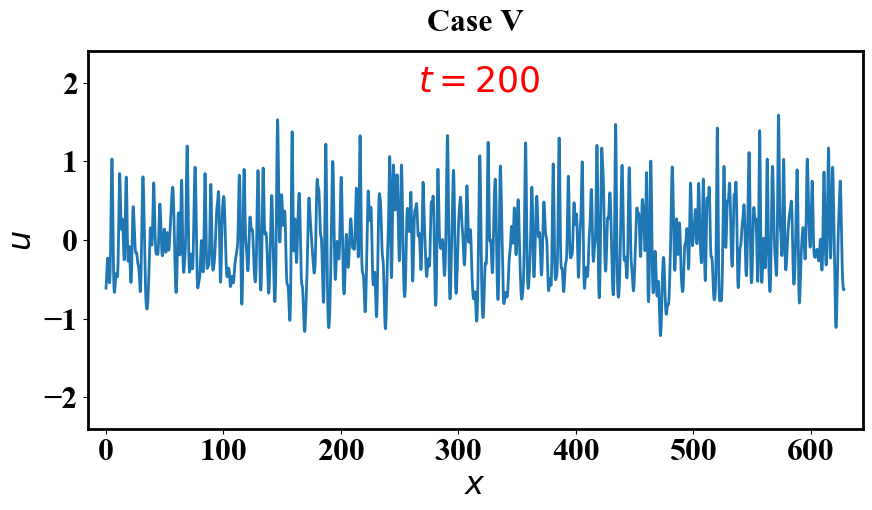}
	}\\
	\subfigure{
		\includegraphics[width=0.485\textwidth,height=0.25\textwidth]{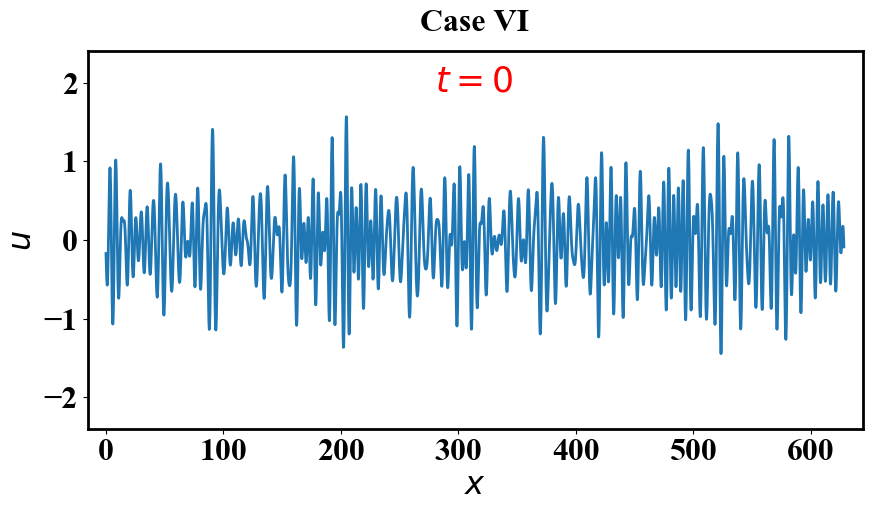}
		\includegraphics[width=0.485\textwidth,height=0.25\textwidth]{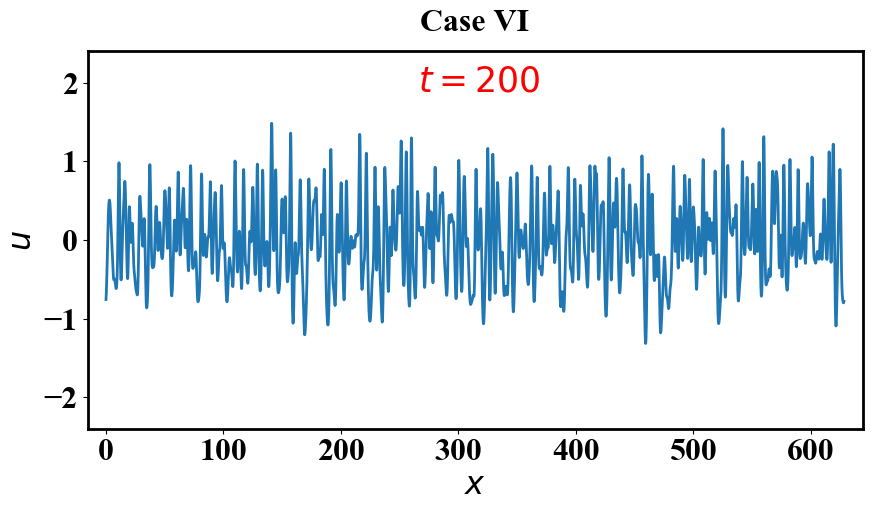}
	}
	\caption{{\bf Subsection} \ref{eg:RBW}: The profiles of numerical solution $u$ that computing by the EIP-QAV-EPRK-2 scheme. The time step is $\Delta t = 0.01$. The snapshots are taken at $t = 0$ and $t = 200$ for Case IV, Case V and Case VI.\label{fig:456cases-u0}}
\end{figure}

\begin{figure}[H]
	\centering
	\subfigure{
		\includegraphics[width=0.323\textwidth,height=0.25\textwidth]{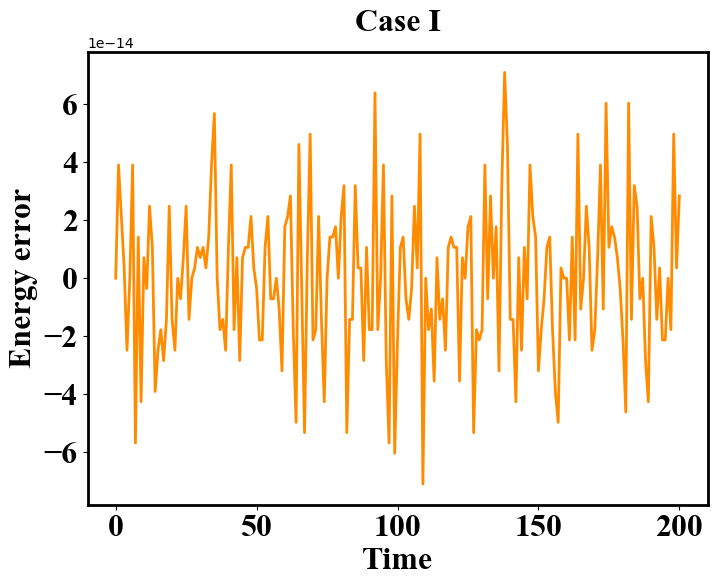}
		\includegraphics[width=0.323\textwidth,height=0.25\textwidth]{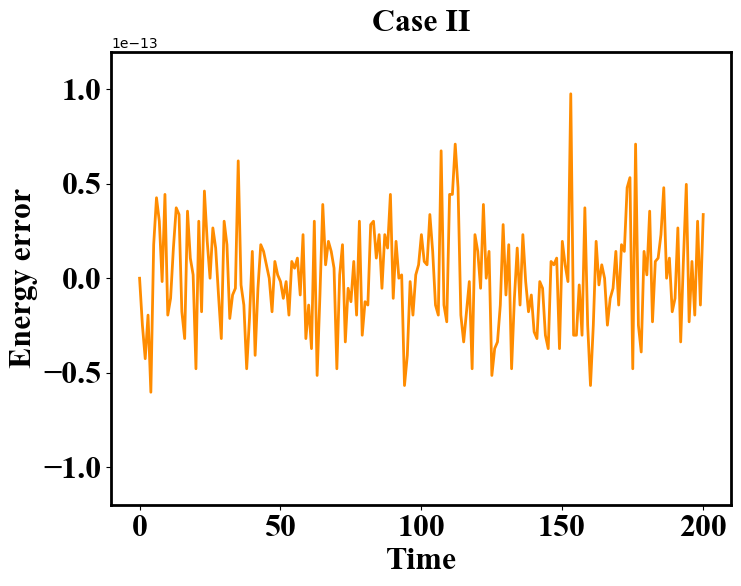}
		\includegraphics[width=0.323\textwidth,height=0.25\textwidth]{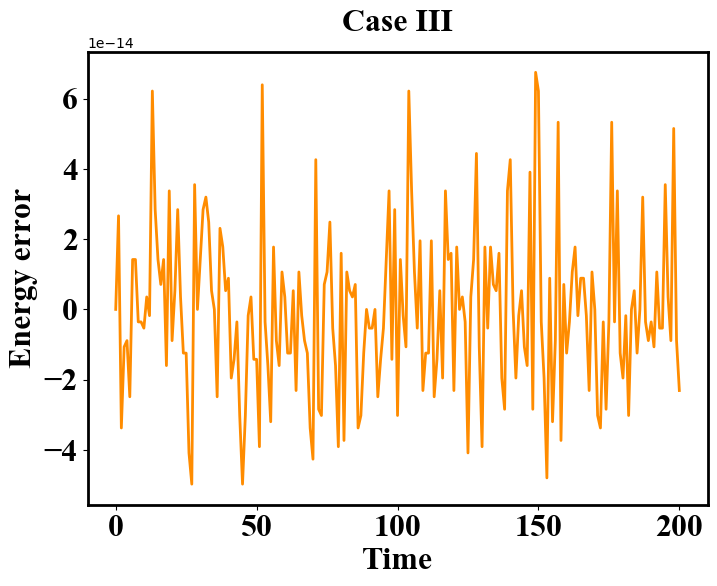}
	}
	\caption{{\bf Subsection} \ref{eg:RBW}: Time evolutions of the original energy errors for Case I, Case II and Case III.\label{fig:123cases-energy}}	
\end{figure}

\begin{figure}[H]
	\centering
	\subfigure{
		\includegraphics[width=0.323\textwidth,height=0.25\textwidth]{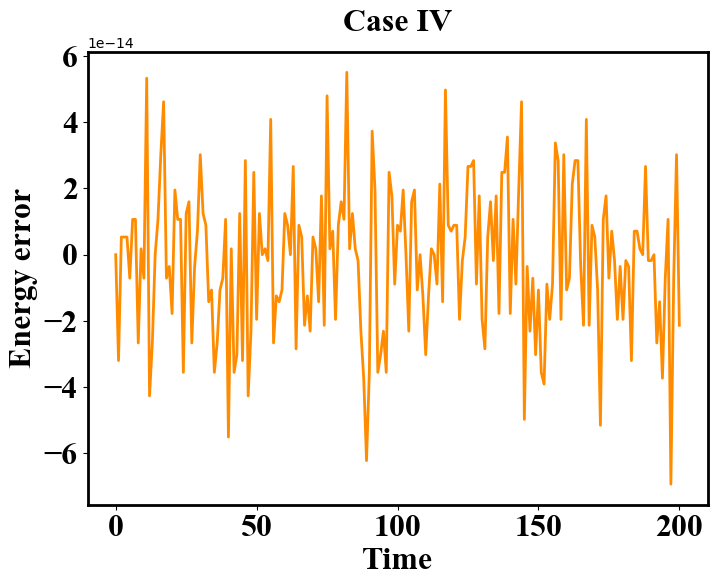}
		\includegraphics[width=0.323\textwidth,height=0.25\textwidth]{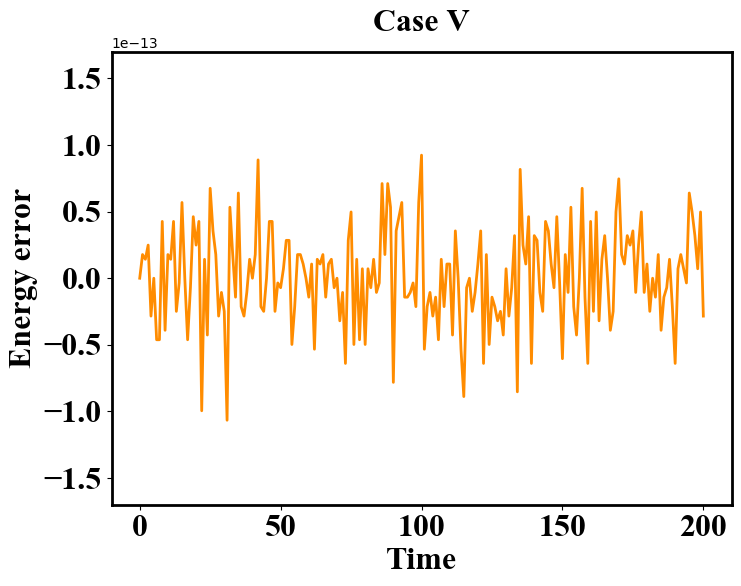}
		\includegraphics[width=0.323\textwidth,height=0.25\textwidth]{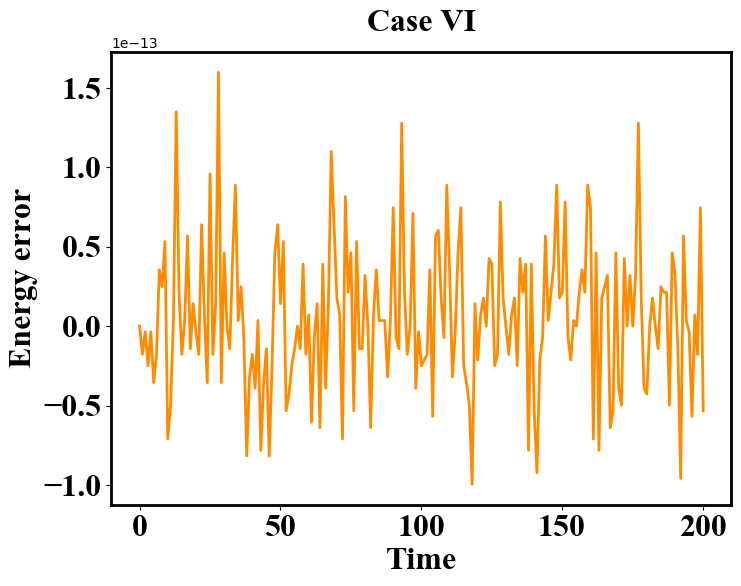}
	}
	\caption{{\bf Subsection} \ref{eg:RBW}: Time evolutions of the original energy errors for Case IV, Case V and Case VI.\label{fig:456cases-energy}}	
\end{figure}

\section{Conclusion}
In this paper, we have proposed a new technique to construct arbitrarily high-order energy-preserving algorithms for the KdV equation. It consists of two important steps, namely the QAV reformulation and the symplectic RK method. Based on our theory, a special class of RK methods can be applied directly to develop arbitrarily high-order energy-preserving algorithms for conservative systems with general polynomial energy of degree greater than 2. Different from the IEQ and SAV approaches, the proposed QAV-EPRK method can eliminate the introduced auxiliary variable and conserve the original energy conservation law. Compared with the existing high-order energy-preserving methods, our QAV-EPRK schemes are based on the traditional RK theory and do not require integrals. Numerical tests are presented to confirm the theoretical analysis and illustrate the usefulness and efficiency of the proposed schemes. It is worthwhile to emphasize that the numerical strategy presented in this paper can be generalized for conservative systems with general polynomial energy or some non-polynomial cases, which will be further discussed in our future work.

\section*{Acknowledgment}
Yuezheng Gong's work is partially supported by the Foundation of Jiangsu Key Laboratory for Numerical Simulation of Large Scale Complex Systems (Grant No. 202002), the Natural Science Foundation of Jiangsu Province (Grant No. BK20180413) and the National Natural Science Foundation of China (Grants No. 11801269, 12071216). Chunwu Wang's work is partially supported by Science Challenge Project (Grant No. TZ2018002). Qi Hong's work is  partially supported by the China Postdoctoral Science Foundation (Grant No. 2020M670116), the Foundation of Jiangsu Key Laboratory for Numerical Simulation of Large Scale Complex Systems (Grant No. 202001).

%

\end{document}